\documentclass{siamart0516}
\usepackage{geometry}                
\usepackage{graphicx}
\usepackage{amssymb, amsmath}
\usepackage{epstopdf}
\usepackage{tikz}
\usepackage{pgfplots}
\pgfplotsset{compat=newest}
\newcommand{\bbR}{\ensuremath{\mathbb{R}}}

\newcommand{\bfm}{\ensuremath{\mathbf{m}}}

\newcommand{\bfv}{\ensuremath{\mathbf{v}}}

\newcommand{\bfx}{\ensuremath{\mathbf{x}}}
\newcommand{\bfz}{\ensuremath{\mathbf{z}}}

\newcommand{\mbar}{\ensuremath{\bar{m}}}

\newcommand{\oneTo}[1]{\ensuremath{ \{1, \ldots, #1 \} }}

\def \etal {\emph{et al.}}

\newcommand{\beq}{\begin{equation}}
\newcommand{\eeq}{\end{equation}}
\newcommand{\bal}{\begin{align}}
\newcommand{\eal}{\end{align}}

\newcommand{\mb}{\ensuremath{\bar{m}}}
\newcommand{\dm}{\ensuremath{\Delta m}}
\newcommand{\dv}{\ensuremath{\Delta v}}
\newcommand{\vb}{\ensuremath{\bar{v}}}

\DeclareMathOperator{\lchild}{lc}
\DeclareMathOperator{\rchild}{rc}
\DeclareMathOperator{\ldesc}{ld}
\DeclareMathOperator{\rdesc}{rd}
\DeclareMathOperator{\desc}{d}
\DeclareMathOperator{\parent}{p}

\title{Dynamical, value-based decision making among $N$ options:\\ a constructive approach to unfolding\\ the symmetric pitchfork bifurcation} 
\author{Paul Reverdy\thanks{P. Reverdy is with the Department of Aerospace and Mechanical Engineering, University of Arizona, Tucson, AZ 85721. Email: \texttt{preverdy@arizona.edu}}}

\begin{document}
\maketitle

\begin{abstract}
Decision making is a fundamental capability of autonomous systems. As decision making is a process which happens over time, it can be well modeled by dynamical systems. Often, decisions are made on the basis of perceived values of the underlying options and the desired outcome is to select the option with the highest value. This can be encoded as a bifurcation which produces a stable equilibrium corresponding to the high-value option. When some options have identical values, it is natural to design the decision-making model to be indifferent among the equally-valued options, leading to symmetries in the underlying dynamical system. For example, when all $N$ options have identical values, the dynamical system should have $S_N$ symmetry. Unfortunately, constructing a dynamical system that unfolds the $S_N$-symmetric pitchfork bifurcation is non-trivial. In this paper, we develop a method to construct an unfolding of the pitchfork bifurcation with a symmetry group that is a significant subgroup of $S_N$. The construction begins by parsing the decision among $N$ options into a hierarchical set of $N-1$ binary decisions encoded in a binary tree. By associating the unfolding of a standard $S_2$-symmetric pitchfork bifurcation with each of these binary decisions, we develop an unfolding of the pitchfork bifurcation with symmetries corresponding to isomorphisms of the underlying binary tree.
\end{abstract}

\begin{keywords}
  symmetries of dynamical systems, bifurcations, binary trees
\end{keywords}

\begin{AMS}
  37C80, 37G10, 37E25
\end{AMS}

A fundamental characteristic of systems that exhibit autonomy is the ability to decide among a set of possible actions. Such systems arise in the natural world, composed of individuals or groups of humans or animals, and also in the artificial world, composed of robots or other algorithmic agents. Studying decision making in both types of systems is of scientific and engineering interest. In natural systems, one tends to focus on developing models that explain observed behavior, while in artificial systems, one tends to focus on developing models that perform some desired behavior. In both cases, it is natural to develop models which make decisions using a dynamic mechanism that reacts to some perceived notion of action value. This paper focuses on the case of a single agent and an arbitrary number of actions and constructs such a mechanism based on the pitchfork bifurcation with symmetries corresponding to interchange of action labels.

As done in numerous recent papers \cite{TDS-etal:12, JARM-etal:09, DP-etal:13, AR-etal:17,RG-etal:18, AF-MG-NEL:19}, it is natural to model the decision making process using a dynamical system. Decision making is itself a dynamic process, affected by the set of actions available to the individual and the context, including the individual's internal state and external inputs. 
For example, consider an animal running in a rough environment. The actions available to the animal correspond to locations it can place its feet on each stride, the internal state to the animal's level of fatigue and fear, and the external inputs to stimuli revealing the quality of available foot placements or the presence of a predator. Decisions must be made quickly, since one is required for each step, and the decision-making process must be flexible, so that decisions can be revised smoothly if a previously attractive location turns out to be undesirable or the appearance of a predator requires a startle response. As argued in \cite{AF-MG-NEL:19}, these requirements suggest viewing the decision-making process as a dynamical system operating in continuous time and on a continuous state space.

The example of the running animal is an example of the ecological theory of \emph{affordances}. An affordance is an opportunity for action which the environment affords the agent \cite{JJG:79}, and the theory of affordances has gained significant traction as a model of ecological decision making \cite{KF-etal:13, NFL-GP:15}. In particular, the so-called \emph{affordance competition hypothesis} \cite{PC-JFK:10, PC:07} suggests that animals perform physical behaviors by continuously identifying affordances and selecting among them by a process which is biased by the desirability of their predicted outcomes. More recently, affordance theory has been investigated in robotics as an approach for developing novel control laws \cite{APD-LPK-WHW:98, PZ-etal:17}. In our own work we have begun to use dynamical systems which we call \emph{motivation dynamics} to develop both theoretical \cite{PBR-DEK:18} and practical \cite{PBR-VV-DEK:20} robot control systems. Therefore, a dynamical system for making decisions on the basis of perceived action values could be valuable for applications in both ecology and robotics.

This paper constructs a continuous dynamical system that allows a single agent to make decisions on the basis of perceived action values. Motivated by recent work in this area \cite{TDS-etal:12,DP-etal:13,AF-MG-NEL:19,RG-etal:18}, we base our dynamical system on the pitchfork bifurcation with certain symmetries. In the nominal case where all $N$ actions have the same value, we require that the dynamical system be equivariant under interchange of actions and impose symmetries which correspond to permutations of the actions. This allows us to use the tools of equivariant bifurcation theory \cite{MG-DGS:85,MG-IS:03}. We introduce asymmetric action values as unfolding parameters of this pitchfork bifurcation, and construct the system so that decisions correspond to bifurcations.

The key difficulty in our approach lies in the fact that developing an unfolding for the pitchfork bifurcation  with $S_N$-symmetry is non-trivial. In the case of a binary decision between $N=2$ options, i.e., $S_2$ symmetry, the unfolding properties are well understood \cite{MG-DGS:85}, but generalizing to the case $N>2$ is complex. For $N=2$, the model of value-based decision making introduced by \cite{TDS-etal:12} and analyzed by \cite{DP-etal:13} can be shown to embed an unfolding of the pitchfork bifurcation \cite{PBR:19d}, but this structure does not naturally generalize even to the case $N=3$ \cite{AR-etal:17}. Reina \etal~\cite{AR-etal:17} show how to modify the Seeley \etal~model to recover desirable bifurcation properties in the case of an arbitrary number of options $N$ but with only one best option, though they do not provide analysis for their general case. In the recent paper by Franci \etal~\cite{AF-MG-NEL:19} the authors use equivariant bifurcation theory to analyze a similar system in the general case of $N>2$ options (and indeed, for a generic number $M>1$ of interacting agents). However, Franci \etal~present explicit equations only for the case of $N=3$ options and $M=3$ agents, and they consider only the symmetric case of equally-valued options.

In contrast to these preceding works, the present paper presents an explicit construction of an unfolding of a pitchfork bifurcation with a symmetry group that is a significant subgroup of $S_N$ for $N\geq2$. The core idea underlying our construction is that one can parse a decision among $N \geq 2$ options into a series of binary decisions; these decisions can be represented by a binary tree. This representation allows us to take advantage of the well-understood representation of binary decisions in terms of the unfolding of a standard pitchfork singularity. We then construct a dynamical system with the desired bifurcation for $N \geq 2$ by assembling a series of standard pitchfork systems in a way that reflects the parsing encoded by the tree structure. In the symmetric case where the options are equally-valued, the dynamical system is equivariant under transformations (i.e., relabeling of options) that correspond to isomorphisms of the binary tree. Breaking the symmetry of option values naturally results in an unfolding of the pitchfork bifurcation.

The contributions of this paper are three-fold. First, we develop a dynamical systems model of decision making where the structure of the decision, and of the dynamical system itself, is encoded in a tree graph structure. Such graph structures have often been used in studying multi-agent decision-making problems, \cite{ROS-RMM:04, RG-etal:18, AF-MG-NEL:19}, but to our knowledge this is the first time a graph structure has been used to organize the dynamical decision-making procedure of a single agent. Secondly, the dynamical system we develop has a novel feedforward structure due to its recursive definition, where the vector field is constructed by recursively parsing down the tree structure. The structure is such that dynamics flow from the root node towards the leaves, which represent options. Thirdly, in building our model we develop a constructive approach to unfolding the equivariant pitchfork bifurcation. Our model has symmetries which correspond to isomorphisms of the underlying binary tree structure, which results in our pitchfork having a symmetry group that is a significant subgroup of $S_N$. This significantly extends existing results, e.g., in \cite{AR-etal:17} and \cite{AF-MG-NEL:19}.

\section{Dynamical systems as models of $N$-ary decision making}

In this section we summarize our requirements for dynamical systems models of decision making. We suppose we have $N \geq 2$ actions. In the following, we refer to actions by the more generic term \emph{option} to reflect that decision making need not be directly tied to action. The $i^{th}$ option has value $v_i > 0$. Inspired by previous work in the decision making literature \cite{TDS-etal:12,DP-etal:13}, we encode the decision state of the system in a variable $m \in \Delta^N$, where
\[ \Delta^N = \left\{ x \in \bbR^{N+1} \Big| x_i \geq 0, \sum_i x_i = 1 \right\} \]
denotes the $N$-simplex. Let $e_i \in \bbR^{N+1}, i \in \oneTo{N+1},$ denote the $i^{th}$ corner of $\Delta^N$, i.e., the vector with the $i^{th}$ element equal to one and all other elements equal to zero. The $i^{th}$ element of $m$, denoted $m_i$, represents the degree to which the system has committed to option $i$. Note that $m_i \in [0,1]$. The $N+1$ element of $m$ represents the degree to which the system is uncommitted to any option, and the normalization condition ensures that the total commitment is finite.

We encode the decision-making process in a continuous vector field $f$ operating on the state space $\Delta^N$. A decision for option $i$ is represented by the state $m$ approaching the corresponding corner of the $N$-simplex, i.e., the state $e_i$. See Figure \ref{fig:3Simplex} for the case $N=2$. We want the system to commit to an option $i$ when its corresponding value $v_i$ is sufficiently high. Thus, we construct the vector field $f$ to have an attracting equilibrium near $e_i$ when $v_i$ is high. When several different options have high values but there is no clear maximum, it is desirable to have several equilibria that encode partial commitment to the attractive options.

\begin{figure}[ht]
\centering
\begin{tikzpicture}
\begin{axis}[width=0.75*\textwidth,
                axis lines=middle,
                axis equal,
                xlabel=$x_1$,
                ylabel=$x_2$,
                zlabel=$x_3$,
                xmin=0,
                xmax=1.25,
                ymin=0,
                ymax=1.25,
                zmin=0,
                zmax=1.25,
                xtick={0,1},
                ytick={0,1},
                ztick={0,1},
                view={135}{30}]
\addplot3[patch, color=blue!25, fill opacity=0.25, faceted color=black, line width=0.5pt] coordinates{(1,0,0) (0,1,0) (0,0,1)};
\end{axis}
\end{tikzpicture}
\caption{Plot of the 2-simplex $\Delta^2$. The simplex is the shaded section of the plane $x_1 + x_2 + x_3 = 1$.}
\label{fig:3Simplex}
\end{figure}
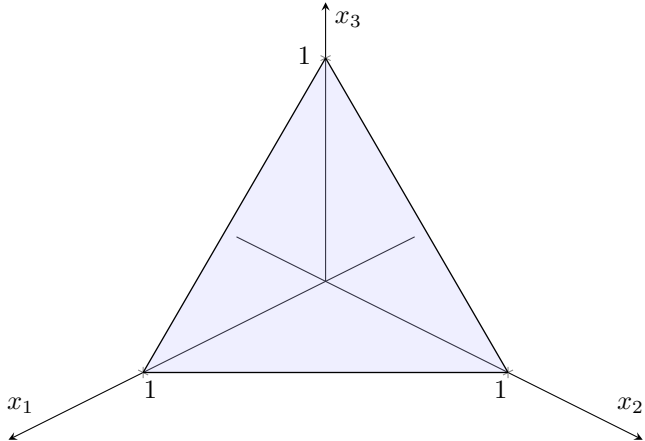

To structure the equilibria, we require symmetry in the vector field $f$ that reflects the symmetry in the options. We model options to be distinguished only by their associated values. Thus symmetries in the options are permutations of the set $\{v_i, i \in \oneTo{N} \}$ that leave the set invariant. In the fully-symmetric case where $v_i = v > 0 \forall i \in \oneTo{N}$, all options are equally valued and therefore identical, and the symmetry group is the symmetric group $S_N$. When the option values are not equal, we want the symmetry in $f$ to be broken. Regardless of any symmetries that may be present, we want the system to remain uncommitted when the values $v_i$ are low and to commit when the values are above a threshold. In the symmetric case, we want the system to remain uncommitted when the value $v$ is low, and to commit symmetrically to all options when $v$ increases beyond a threshold. Conversely, in the asymmetric case, we want the system to commit asymmetrically to the high-value options.

\subsection{Equivariant bifurcation theory}
The desire for a switching commitment process strongly suggests the use of a bifurcation in the dynamical system, and the symmetry requirement encourages the use of equivariant bifurcation theory \cite{MG-DGS:85,MG-IS-DGS:88}. In the case of $N=2$ options, our requirements are satisfied by selecting the vector field $f$ to embed a pitchfork bifurcation as shown in Figure \ref{fig:pitchfork}. The mechanism can be viewed as follows. Let $x = m_1 - m_2$ represent the degree of commitment for option 1 over option 2. Then, in the symmetric case, selecting $f$ to be a pitchfork bifurcation, i.e.,
\beq \label{eq:pitchfork}
\dot x = f(x, \mu) = x(\mu - x^2), \ \mu \in \bbR
\eeq
will yield the desired behavior. Note that $f$ is an odd function of $x$, so the system \eqref{eq:pitchfork} obeys the $S_2$ symmetry corresponding to switching the option labels. When $\mu$ (which represents the value $v$) is less than zero, the unique equilibrium of \eqref{eq:pitchfork} is $x=0$, corresponding to equal commitment to both options, and this equilibrium is stable. As $\mu$ increases through zero, the equilibrium at $x=0$ becomes unstable and two new stable equilibria appear at $\pm \sqrt{\mu}$. These facts are summarized in the bifurcation diagram shown in Figure \ref{fig:pitchfork}\textbf{(a)}. Equivariant bifurcation theory predicts that these two branches must appear symmetrically in the post-bifurcation regime; each branch represents a commitment to one option or the other. Due to symmetry, commitment to either option is possible; the option to which the system commits is determined by initial conditions.

The case of asymmetric option values is naturally modeled by breaking the symmetry of the vector field $f$. The study of such symmetry breaking is a core part of equivariant bifurcation theory, and the fundamental concept is the so-called \emph{unfolding} of the bifurcation. One begins with the concept of a \emph{perturbation} of a bifurcation. Specifically, a perturbation $F$ of a bifurcation $f$ is a function $F(x, \mu, \alpha)$ such that $F(x, \mu, 0) = f(x, \mu)$. A \emph{universal unfolding} is a $k$-parameter (i.e., $\alpha \in \bbR^k$) perturbation $F$ such that any small perturbation of $f$ can be expressed in terms of the $k$ parameters that define $F$. The two-parameter family
\beq \label{eq:pitchfork-unfolding} 
F(x, \mu, \alpha) = x(\mu - x^2) + \alpha_1 + \alpha_2 x^2, \ \alpha_1, \alpha_2 \in \bbR
\eeq
is a universal unfolding of the pitchfork bifurcation \eqref{eq:pitchfork} \cite{MG-DGS:85}. As the parameters $\alpha_1, \alpha_2$ are varied, the bifurcation diagram for $\dot x = F(x, \mu, \alpha)$ varies as well. In particular, for appropriate parameter values, the system has a single equilibrium in the post-bifurcation regime, and this equilibrium is both nonzero and stable (see Figure \ref{fig:pitchfork}\textbf{(b)}). Thus, by choosing appropriate values for the unfolding parameters, one can encode a globally-attracting preference for one option or the other. To complete the model for $N=2$ options, it remains to relate the option values $v_i$ to the bifurcation parameter $\mu$ and unfolding parameters $\alpha_i$ in \eqref{eq:pitchfork-unfolding}.

\begin{figure}[ht]
\centering
\begin{tabular}{cc}
\includegraphics[width=0.47\textwidth]{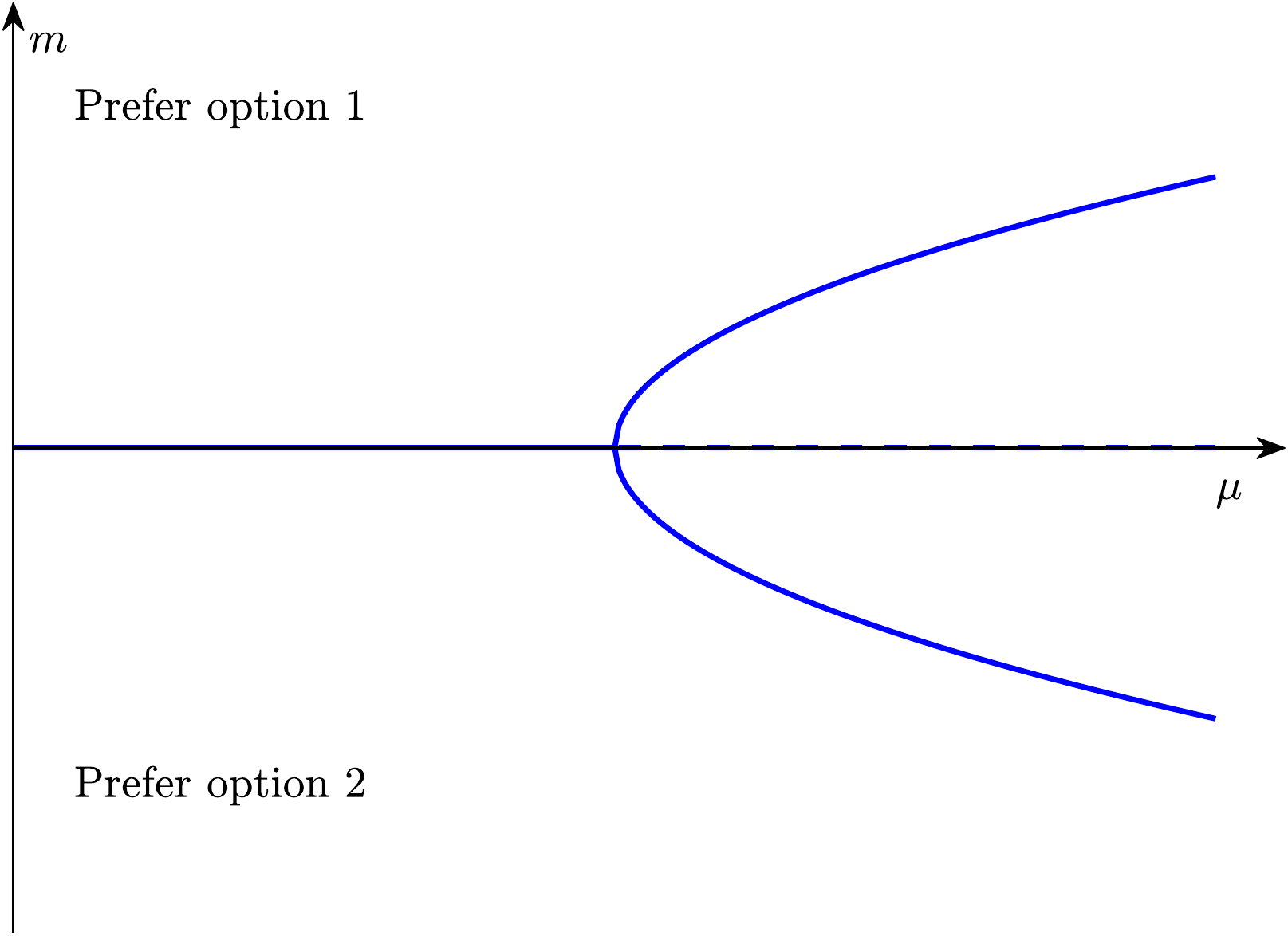} & \includegraphics[width=0.47\textwidth]{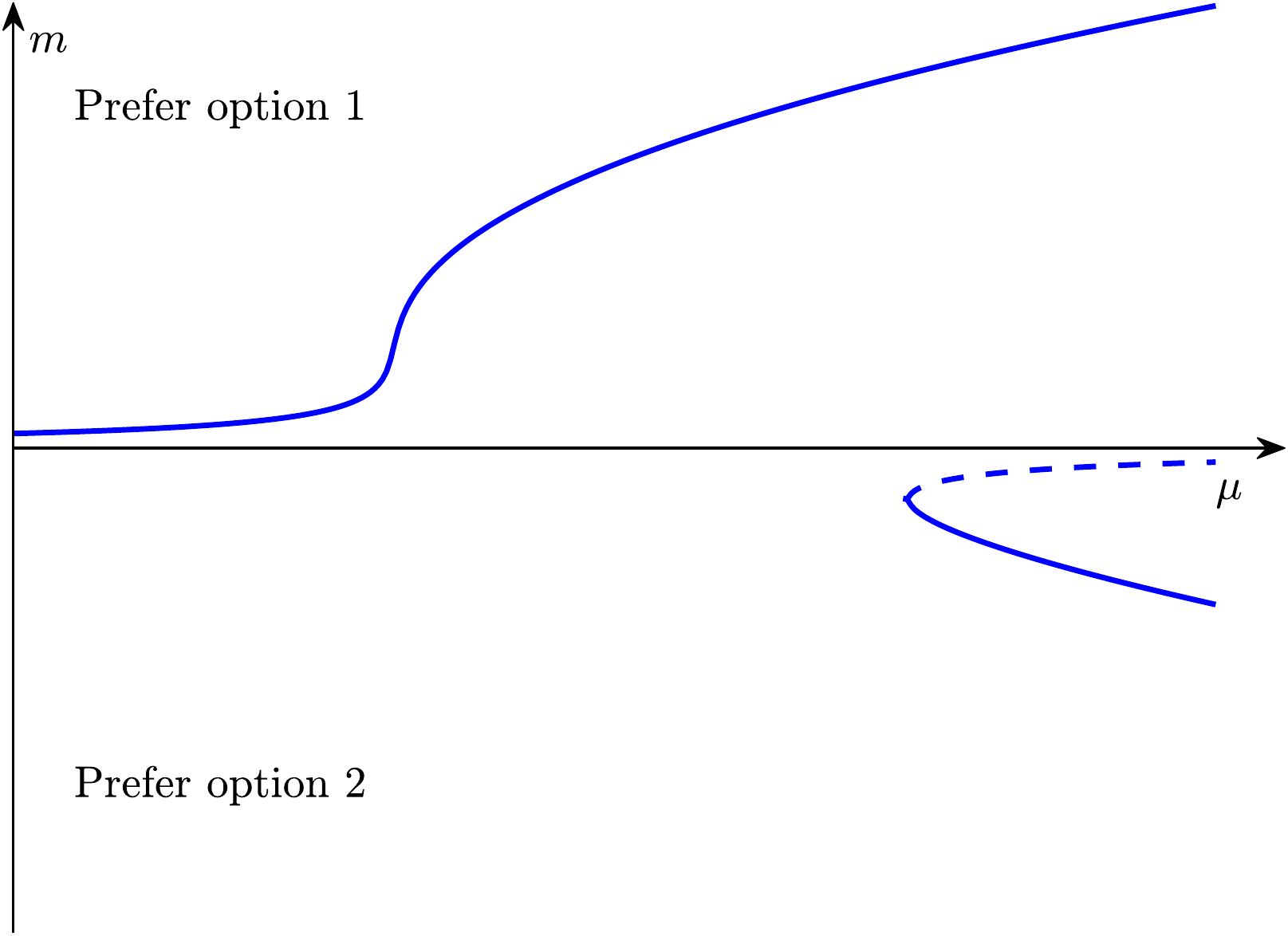} \\
\textbf{(a)} & \textbf{(b)}
\end{tabular}
\caption{Bifurcation diagrams for \textbf{(a)} the symmetric pitchfork \eqref{eq:pitchfork} and \textbf{(b)} the unfolded pitchfork \eqref{eq:pitchfork-unfolding}. Stable equilibria are represented by solid lines and unstable equilibria by dashed lines. Equilibria above the $\mu$ axis represent a preference for option 1, while those below represent a preference for option 2. In panel \textbf{(a)} the system has a single stable equilibrium representing no preference when $\mu$ is small and two symmetric stable equilibria representing preferences for option 1 and option 2 respectively, when $\mu$ is large. Note particularly in panel \textbf{(b)} that, for intermediate values of $\mu$, the system has a single stable equilibrium representing a preference for option 1.}
\label{fig:pitchfork}
\end{figure}

\subsection{Seeley \etal~model for $N=2$}
In their paper \cite{TDS-etal:12}, Seeley \etal~study the value-sensitive decision-making problem for $N=2$ options and develop a dynamical systems model. Their model can be shown to embed an unfolded pitchfork \cite{PBR:19d}, thus completing the model whose mechanism we laid out in the preceding section. Concretely, Seeley \etal~let the state of their model be $m = (m_1, m_2, m_U) \in \Delta^2$ and set $\dot m = f(m, v)$, with
\begin{align} \label{eq:miDynamics} 
\dot{m}_i = v_i m_U - m_i \left( \frac{1}{v_i} - v_i m_U + \sigma (1 - m_i - m_U) \right)
\end{align}
for each $i = 1,2$, and the dynamics of $m_U$ are determined by the normalization constraint. As above, $v_i$ denotes the value of option $i$, and $\sigma > 0$ is a constant parameter.

In the symmetric case of $v_1=v_2 = v$ the dynamics \eqref{eq:miDynamics} obey an $S_2$ symmetry. Seeley \etal~\cite{TDS-etal:12} showed that, in the symmetric case, the dynamics \eqref{eq:miDynamics} exhibits a pitchfork bifurcation as $v$ and $\sigma$ increase above a threshold. In the pre-bifurcation regime, the system has a single stable equilibrium with the symmetry $m_1 = m_2$. Because of the symmetry, the system does not commit to either option, and the equilibrium is said to be a \emph{deadlock}. In the post-bifurcation regime, the deadlock equilibrium is unstable and two additional equilibria emerge, each representing a decision to commit to one of the two options. In previous work \cite{TDS-etal:12}, the parameter values at the bifurcation point were found to satisfy
\beq \label{eq:sigmaCrit}
\sigma = \frac{4v^3}{(v^2-1)^2}.
\eeq
Note that either $\sigma$ or $v$ can be interpreted as the bifurcation parameter, with \eqref{eq:sigmaCrit} defining the bifurcation value. In other words, fixing $\sigma$, the bifurcation occurs as $v$ increases through a threshold, while fixing $v$, the bifurcation occurs as $\sigma$ increases.

In the asymmetric case of $v_1 \neq v_2$, the $S_2$ symmetry of the dynamics is broken and the pitchfork bifurcation unfolds as studied by Pais \etal~\cite{DP-etal:13}. For a fixed value of $\sigma$, the number of equilibria of \eqref{eq:miDynamics} depends on the parameters $v_1$ and $v_2$. Certain parameter values result in a single stable equilibrium whose location is biased towards the high-value option, while others result in two stable equilibria representing each option and a saddle point in between. The complete phase diagram of the system is complex, but the two primary findings of \cite{DP-etal:13} can be summarized as follows. First, the dynamics \eqref{eq:miDynamics} remain deadlocked (i.e., have a single attractor with $m_1 \approx m_2$) when the average option value $\vb = (v_1 + v_2)/2$ is small. Second, the dynamics decide for the high value option (i.e., for $v_1 > v_2$, have a single attractor with $m_1 \gg m_2$) when the difference in option values $\dv = v_1-v_2$ is sufficiently large relative to $\vb = (v_1 + v_2)/2$. In symbols, we have that the system makes a decision when 
\[ \frac{| \dv |}{\vb} > \kappa (\sigma), \]
where $\kappa(\sigma)$ is a coefficient that depends on the parameter $\sigma$. Pais \etal~\cite{DP-etal:13} note that this behavior is analogous to Weber's law of just-noticeable differences from psychology, which states that the minimum difference in stimulus intensity required to discriminate between two different stimuli varies linearly with their mean intensity.

The implication of these two findings is that the decision-making dynamics \eqref{eq:miDynamics} has several desirable properties. First, when both options are poor (corresponding to a low value of $\vb$), the system remains deadlocked and avoids making a decision, e.g., to wait for more information. When at least one option is sufficiently satisfactory (corresponding to a high value of $\vb$), the system will quickly commit to an option, and preferentially select the one with a higher value. These are properties that we seek to generalize to the case of $N>2$ options.

\subsection{Reduction of the Seeley \etal~model}
As shown in the preceding sections, the Seeley \etal~model \eqref{eq:miDynamics} has desirable characteristics that we seek to generalize. However, the functional form of \eqref{eq:miDynamics} obscures the unfolding of the pitchfork bifurcation which serves as the fundamental decision-making mechanism. In recent work \cite{PBR:19d}, we studied the dynamics \eqref{eq:miDynamics} using model reduction techniques to elucidate the unfolding.

The model reductions studied in \cite{PBR:19d} use singular perturbation theory. Specifically, the reduction approach maps $v_i \mapsto K v_i$ for a constant gain $K > 0$ and takes the singular limit $K \to +\infty$. This approach is similar to an analysis performed in \cite{DP-etal:13}, where the authors studied the limit $\vb \to +\infty$; however, the approach using the gain $K$ preserves the relative difference in values $\dv/\vb$. This ratio is key in defining the unfolding of the pitchfork bifurcation embedded in \eqref{eq:miDynamics}.

The bifurcation is more readily analyzed by expressing $m \in \Delta^2$ in terms of mean-difference coordinates defined by
\[ \dm = m_1 - m_2, \mb = \frac{m_1 + m_2}{2} \]
which are analogous to the definitions of $\dv$ and $\vb$ made above. Note that the definitions of these new coordinates and the definitions of $\Delta^2$ and $(v_1, v_2) \in \bbR_+^2$ imply that $\mb, \vb > 0$ and that $-2\mb \leq \dm \leq 2 \mb$ and $-2 \vb < \dv < 2 \vb$.

In the mean-difference coordinates, the dynamics \eqref{eq:miDynamics} of $m = ({\dm}, {\mb})$ take the form
\begin{align} \label{eq:DeltamDot}
\dot{\dm} &= f_{\dm}(\dm, \mb; \vb; \dv)\\ &= - \left( \frac{2 \mb + \dm}{K(2 \vb + \dv)} -\frac{2 \mb - \dm}{K(2 \vb - \dv)} \right) \nonumber + K \vb \dm (1-2 \mb) + K \dv (1-2 \mb) (1+ \mb),
\end{align}

\begin{align} \label{eq:mBarDot}
\dot{\mb} =& f_{\mb}(\dm, \mb; \vb, \sigma; \dv)\\ =& \frac{1}{2} \biggl( -\frac{2 \mb + \dm}{K(2 \vb + \dv)} - \frac{2 \mb - \dm}{K(2 \vb - \dv)} + \frac{K(2 \vb + \dv)}{2}(1-2\mb) (1 + \frac{2 \mb + \dm}{2}) \nonumber \\ 
 & \ \ \ \ \ + \frac{K(2 \vb - \dv)}{2}(1-2\mb) (1 + \frac{2 \mb - \dm}{2}) - \frac{\sigma}{2} (2 \mb + \dm)(2 \mb - \dm) \biggr). \nonumber
\end{align}

A straightforward application of singular perturbation theory with small parameter $\epsilon = 1/K$ and coordinates $x = \dm$, and $y = (1-2 \mb)/\epsilon$ yields the following result.
\begin{theorem}{\cite[Theorem 1]{PBR:19d}} \label{thm:slowDynamics}
In the singular limit $\epsilon \to 0$, the motivation dynamics \eqref{eq:miDynamics} reduce to 
\beq \label{eq:xDotSlow}
\dot x = \frac{\sigma}{2\vb}(1-x^2) \frac{2x + 3 \alpha}{6 + \alpha x},
\eeq
where $\alpha = \dv/\vb$.
\end{theorem}

A standard nonlinear time scaling argument then allows one to eliminate the denominator $6+\alpha x$ from \eqref{eq:xDotSlow} and makes the connection to the unfolding of the pitchfork bifurcation \eqref{eq:pitchfork-unfolding} explicit.
\begin{corollary}{\cite[Corollary 2]{PBR:19d}} \label{cor:singularUnfolding}
The singularly-perturbed motivation dynamics \eqref{eq:xDotSlow} are equivalent to 
\[ x^\prime = x(1-x^2) + \frac{3}{2} \alpha - \frac{3}{2} \alpha x^2, \]
i.e., an unfolding of the pitchfork bifurcation \eqref{eq:pitchfork-unfolding} with bifurcation parameter $\mu \mapsto 1$ and unfolding parameters $\alpha_1 = 3 \alpha/2$ and $\alpha_2 = -3\alpha/2$.
\end{corollary}

The equilibria of the singularly-perturbed system \eqref{eq:xDotSlow} are shown in Figure \ref{fig:equilibria}. Note that the system has three equilibria for all possible values of $\alpha \in [-2,2]$. When $\alpha = 0$ the options are equally valued, the pitchfork is unperturbed, and the equilibria correspond to those of the standard pitchfork \eqref{eq:pitchfork} in the post-bifurcation regime. The equilibrium $x = 0$ is unstable, while those at $x= \pm 1$ are stable. For nonzero $\alpha$ the pitchfork unfolds. In the singularly-perturbed regime, the non-zero equilibria remain at $x = \pm1$, while the intermediate equilibrium shifts to $x = -3\alpha/2$. The intermediate equilibrium is unstable for the values of $\alpha \in [-2/3, 2/3]$ where it exists. The equilibria at $x = \pm 1$ are stable when the value difference $\dv$ is not too biased against the corresponding option. For example, $x=+1$ is a stable equilibrium of \eqref{eq:xDotSlow} for $\alpha \geq -2/3$.

The structure of equilibria shown in Figure \ref{fig:equilibria} determines the decision-making behavior of the model \eqref{eq:miDynamics} in the singular limit. The state $x = \dm = +1$ corresponds to the system committing fully to option 1, i.e., to $m = (1, 0, 0)$. This state is stable for $\alpha = \dv/\vb \geq -2/3$, and globally attracting for $\alpha > 2/3$. In other words, the system can commit to option 1 when when $\dv \geq -2 \vb/3$, and will be globally attracted to committing to option 1 when $\dv > 2 \vb/3$. Switching behavior can occur as $\alpha$ shifts. For example, suppose that the system \eqref{eq:xDotSlow} is initialized with $\alpha < 0$ and $\dm = -0.9$, representing a commitment to option 2. If $\alpha$ is then raised above the value $2/3$, the state $\dm$ will be attracted to the value $+1$, representing a decision to switch and commit to option 1. The rate at which $\dm$ is attracted to $+1$ is governed by the parameter $\sigma$, as can be seen from the form of the dynamics \eqref{eq:xDotSlow}.

We note that the coefficient $3/2$ arises from the last term $K\dv (1-2\mb)(1+\mb)$ in \eqref{eq:DeltamDot} and can be adjusted by changing the coefficient $1+\mb$ to $\beta + \mb$, which changes the coefficient $3/2 = 1 + 1/2$ to $\beta + 1/2$. This observation can be used to control the bifurcation properties of the system, e.g., by making it more or less sensitive to the relative value difference $\alpha$. We do not pursue this line of investigation further in the present paper, but recognize that it is a natural point of departure for further work.

\begin{figure}[htbp]
\vspace{-0.25cm}
\begin{center}
\includegraphics[width=0.7\textwidth]{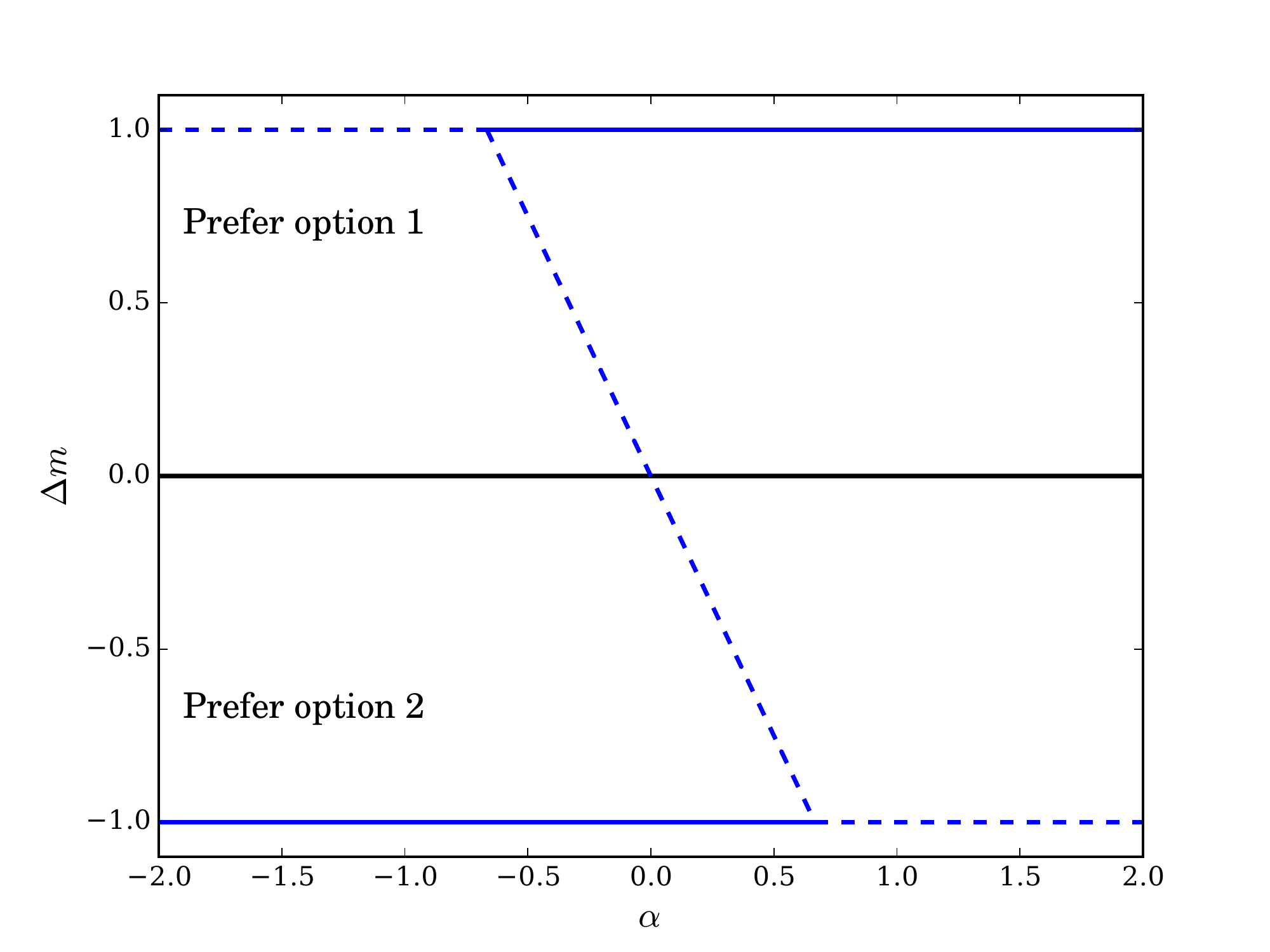}
\caption{Equilibria of the singularly-reduced dynamics \eqref{eq:xDotSlow} as a function of the unfolding parameter $\alpha$. Solid lines represent stable equilibria; dashed lines unstable equilibria. For sufficiently large $|\alpha| > 2/3$, only one equilibrium corresponding to a preference for the high-value option is stable.}
\label{fig:equilibria}
\end{center}
\vspace{-0.7cm}
\end{figure}

In this section, we introduced our requirements for a dynamical system model of value-sensitive decision making We showed how an unfolding of the pitchfork bifurcation can provide the fundamental mechanism for such a model in the case of $N=2$ options, and reviewed a model due to Seeley \etal~that embeds such a pitchfork mechanism along with some recent results reducing that model. In the following section we begin to construct a generalization of the Seeley \etal~model to the case $N>2$ by parsing the decision among $N$ options into a series of binary decisions represented by a tree structure.

\section{Parsing $N$-ary decisions into binary trees}
Inspired by the ideas presented in the previous section, we seek to develop a dynamical systems model of value-sensitive decision-making among $N$ options using an unfolding of the pitchfork bifurcation. The difficulty of this approach is that constructing an unfolding of the pitchfork bifurcation in $N$ dimensions is non-trivial. In this paper, we construct such an unfolding, and thus the desired model, by composing a series of unfoldings of a standard pitchfork bifurcations. The composition is structured by parsing a decision among $N$ options into a series of binary decisions. 

In this section we introduce the idea of parsing a decision among $N$ options into a series of binary decisions represented by a tree structure, and review a number of concepts, primarily from the computer science literature, on binary trees. As an example of an $N$-ary decision, consider the case of a professor who has a variety of tasks and must decide which one to focus on at any given time. She may decide on a task by making a series of binary decisions as shown in Figure \ref{fig:parsing}. For example, she might first decide between working on research or on teaching; given a decision to work on teaching, she may work on preparing a lecture or some assignments. The decision among five options is thus parsed into a series of binary decisions represented by a binary tree. A decision among an arbitrary number $N$ options can be parsed in this way.

We now define a number of terms associated with binary trees. A (rooted) tree is a connected acyclic undirected graph where one node is identified as the root. The \emph{parent} of a node $n$ is the node connected to $n$ on the path to the root, and the \emph{children} of a node $n$ are the nodes for which $n$ is the parent. Similarly, a \emph{descendant} of a node $n$ is any node which is a child of $n$ or is a descendent of any of the children of $n$. A \emph{sibling} of a node $n$ is any other node which shares a parent with $n$. A \emph{leaf} is a node with no children, while an \emph{internal} node is a node which is not a leaf. Finally, a binary tree is a tree where each node has at most two children, referred to as the \emph{left} and \emph{right} children. For such a tree, we refer to the descendants of a node $n$ associated with the right and left children as the \emph{right} and \emph{left} descendants, respectively. Note that when an arbitrary number $N$ of options is parsed into a binary tree $T$, the tree can be selected such that each internal node has two children. Such a tree is referred to as a \emph{full} or \emph{proper} binary tree.

We now formally define a parsing of a decision set.
\begin{definition}[Parsing]
Given a set of $n_o \geq 2$ options, a \emph{parsing} of these options is a proper rooted binary tree $T$ where each leaf node represents an option.
\end{definition}
Often, we will label the nodes with an index $i$. Then, the function $o$ maps leaf nodes to their associated option, i.e., $i \mapsto o(i)$. Figure \ref{fig:parsingLabels} shows the node labels associated with the parsing shown in Figure \ref{fig:parsing}. In this case, we have $o(2) = $ Experiment, $o(3) = $ Code, $o(5) =$ Lecture, $o(7) = $ Homework, and $o(8) =$ Exam.

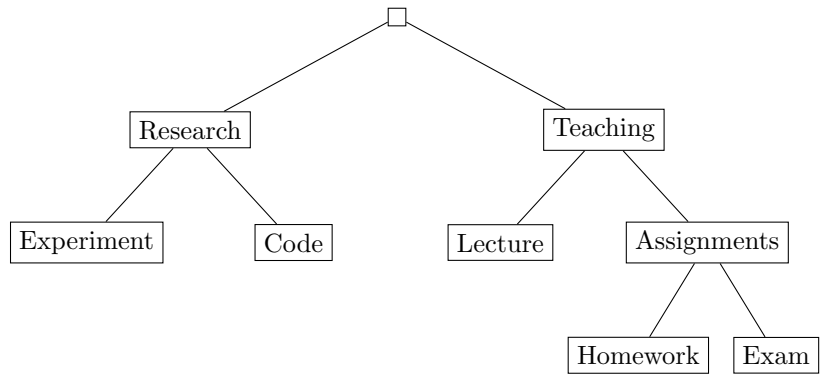
\begin{figure}[ht]
\centering
\begin{tikzpicture}[level/.style={sibling distance=55mm/#1}]
\node [rectangle, draw] (root){}
	child {node [rectangle, draw] (a) {Research}
		child {node [rectangle, draw] (b) {Experiment}}
		child {node [rectangle, draw] (c) {Code}}
	}
	child {node [rectangle, draw] (g) {Teaching}
		child {node [rectangle, draw] (d) {Lecture}}
		child {node [rectangle, draw] (e) {Assignments}
			child {node [rectangle, draw] (f) {Homework}}
			child {node [rectangle, draw] (h) {Exam}}
			}
	};
\end{tikzpicture}
\caption{Parsing a decision among $N$ options into a series of binary decisions.}
\label{fig:parsing}
\end{figure}

\begin{figure}[ht]
\centering
\begin{tikzpicture}[level/.style={sibling distance=55mm/#1}]
\node [circle, draw] (root){0}
	child {node [circle, draw] (a) {1}
		child {node [circle, draw] (b) {2}}
		child {node [circle, draw] (c) {3}}
	}
	child {node [circle, draw] (g) {4}
		child {node [circle, draw] (d) {5}}
		child {node [circle, draw] (e) {6}
			child {node [circle, draw] (f) {7}}
			child {node [circle, draw] (h) {8}}
			}
	};
\end{tikzpicture}
\caption{Node labels associated with the parsing shown in Figure \ref{fig:parsing}. The function $o$ maps leaf nodes to options, e.g., $o(2)=$ Experiment.}
\label{fig:parsingLabels}
\end{figure}
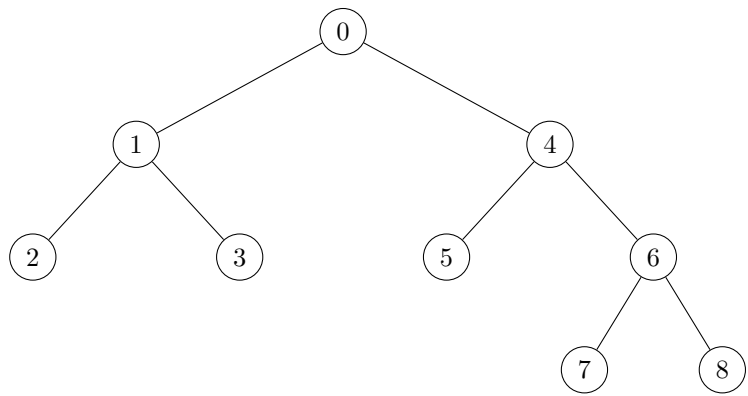

\subsection{Tree traversal}
Tree traversal is a fundamental process acting on a tree, wherein the process visits (and carries out an action on) each node in the tree exactly once. Trees may be traversed in either \emph{depth-first} or \emph{breadth-first} orders, as depicted in Figure \ref{fig:treeTraversal}. As their names imply, depth-first order operates by going as deep down the tree as possible before going to the next sibling, while breadth-first order operates by going through each sibling before going to a descendant. The nodes in Figure \ref{fig:treeTraversal} are labeled according to the order in which they will be visited in depth-first or breadth-first traversal.

\begin{figure}[ht]
\centering
\begin{tabular}{cc}
\begin{tikzpicture}[level/.style={sibling distance=40mm/#1}]
\node [circle, draw] (root){$0$}
	child {node [circle, draw] (a) {1}
		child {node [circle, draw] (b) {2}
			child {node [circle, draw] (f) {3}}
			child {node [circle, draw] (h) {4}}
			}
		child {node [circle, draw] (c) {5}}
	}
	child {node [circle, draw] (g) {6}
		child {node [circle, draw] (d) {7}}
		child {node [circle, draw] (e) {8}}
	};
\end{tikzpicture}
&
\begin{tikzpicture}[level/.style={sibling distance=40mm/#1}]
\node [circle, draw] (root){$0$}
	child {node [circle, draw] (a) {1}
		child {node [circle, draw] (b) {3}
			child {node [circle, draw] (f) {7}}
			child {node [circle, draw] (h) {8}}
			}
		child {node [circle, draw] (c) {4}}
	}
	child {node [circle, draw] (g) {2}
		child {node [circle, draw] (d) {5}}
		child {node [circle, draw] (e) {6}}
	};
\end{tikzpicture}\\
Ordered node list: $(0,1,2,3,4,5,6,7,8)$ & Ordered node list: (0,1,3,7,8,4,2,5,6)\\
Ordered option list: (3,4,5,7,8) & Ordered option list: (7,8,4,5,6) \\
\textbf{(a)} & \textbf{(b)}
\end{tabular}
\caption{Depth-first \textbf{(a)} versus breadth-first \textbf{(b)} traversal of a binary tree. The nodes are labeled with numbers according to the order in which they will be visited during traversal.}
\label{fig:treeTraversal}
\end{figure}
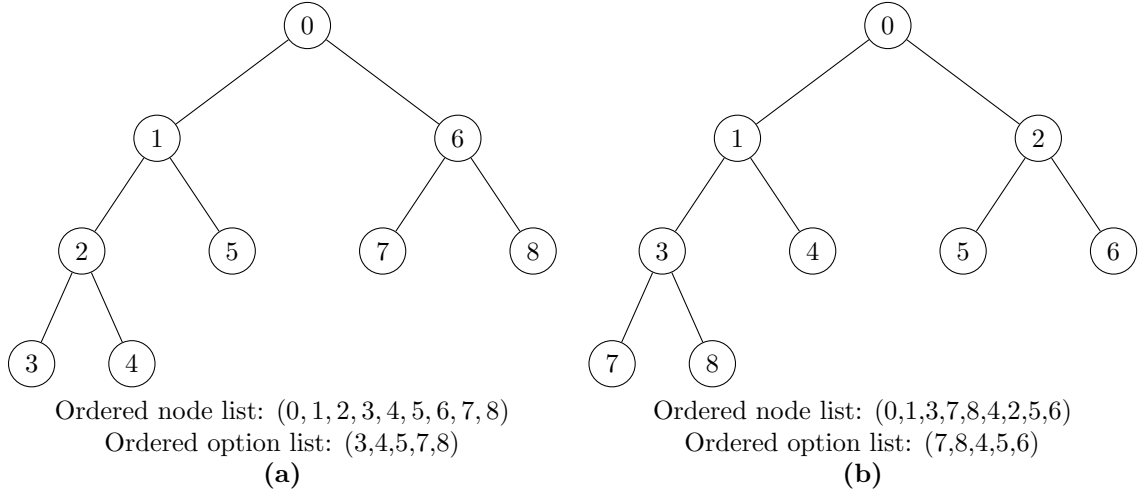

\subsection{Tree paths}
A \emph{path} in a finite tree is a finite sequence of edges which joins a sequence of nodes. For a rooted tree, there is always a unique shortest path from the root to any other node. We denote the sequence of nodes along the shortest path from the root to node $i$ as $p_i$ and denote its $j^{th}$ element as $p_{ij}$. The sequence $p_i$ begins with the root node and ends with the node $i$. The number of nodes in $p_i$ is denoted $|p_i|$.

\subsection{Tree isomorphisms}
Trees may have important symmetries. For example, the nodes of a tree may be rearranged without changing the structure it represents. Two trees which share the same structure are said to be \emph{isomorphic}. The concept of tree isomorphism is inherited from the concept of graph isomorphism \cite{WD-MJD:89}, for which tree isomorphisms are a special case.
\begin{definition}[Rooted tree isomorphism] \label{def:treeIsomorphism}
Let $T_1$ and $T_2$ be two rooted trees with node sets $N_1$ and $N_2$, edge sets $E_1$ and $E_2$, and roots $r_1 \in N_1$ and $r_2 \in N_2$, respectively. An \emph{isomorphism} of $T_1$ and $T_2$ is a bijection between the node sets $\varphi : N_1 \to N_2$ such that
\[ \forall u,v \in N_1 \ \ (u,v) \in E_1 \Leftrightarrow (\varphi(u), \varphi(v)) \in E_2 \]
and $\varphi(r_1) = r_2$.
\end{definition}
In words, a rooted tree isomorphism is a mapping between the node sets such that each edge is preserved, along with the root node. An example, Figure \ref{fig:treeIsomorphism} shows two isomorphisms of the tree presented in Figure \ref{fig:treeTraversal}(a). Note that isomorphisms of binary trees are generated by flips at nodes, wherein the left and right descendants of a given node are exchanged.

Problems associated with tree isomorphisms arise in many applications. In particular, a standard problem in computer science is to determine whether two rooted trees $T_1$ and $T_2$ are isomorphic. A classic algorithm due to Aho, Hopcroft, and Ullman \cite{AVA-JEH-JDU:74} solves the problem in $O(n)$ time for trees with $n$ vertices.

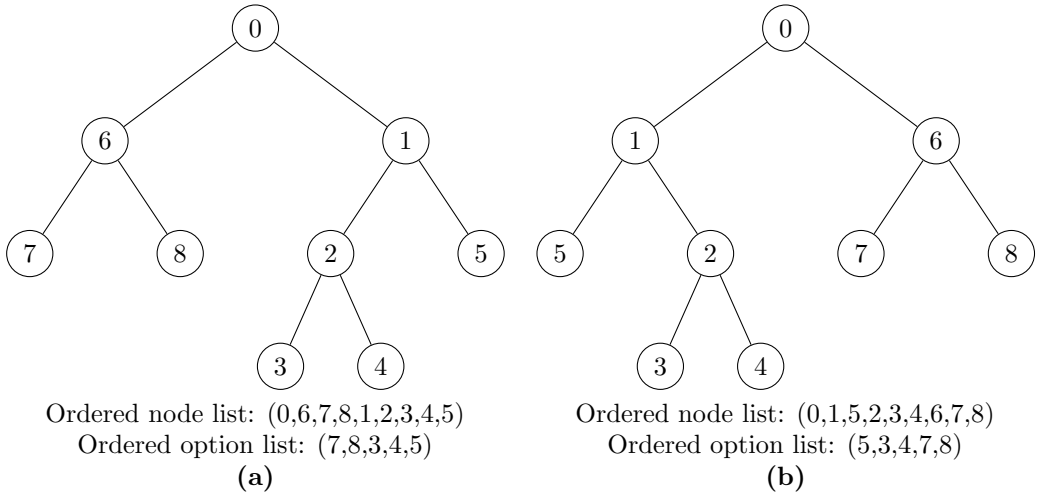
\begin{figure}[ht]
\centering
\begin{tabular}{cc}
\begin{tikzpicture}[level/.style={sibling distance=40mm/#1}]
\node [circle, draw] (root){$0$}
	child {node [circle, draw] (g) {6}
		child {node [circle, draw] (d) {7}}
		child {node [circle, draw] (e) {8}}
	}
	child {node [circle, draw] (a) {1}
		child {node [circle, draw] (b) {2}
			child {node [circle, draw] (f) {3}}
			child {node [circle, draw] (h) {4}}
			}
		child {node [circle, draw] (c) {5}}
	};
\end{tikzpicture}
&
\begin{tikzpicture}[level/.style={sibling distance=40mm/#1}]
\node [circle, draw] (root){$0$}
	child {node [circle, draw] (a) {1}
		child {node [circle, draw] (c) {5}}
		child {node [circle, draw] (b) {2}
			child {node [circle, draw] (f) {3}}
			child {node [circle, draw] (h) {4}}
			}
	}
	child {node [circle, draw] (g) {6}
		child {node [circle, draw] (d) {7}}
		child {node [circle, draw] (e) {8}}
	};
\end{tikzpicture}\\
Ordered node list: (0,6,7,8,1,2,3,4,5) & Ordered node list: (0,1,5,2,3,4,6,7,8)\\
Ordered option list: (7,8,3,4,5) & Ordered option list: (5,3,4,7,8) \\
\textbf{(a)} & \textbf{(b)}
\end{tabular}
\caption{Tree isomorphisms are generated by flips at nodes. Here we show two isomorphisms of the tree presented in Figure \ref{fig:treeTraversal}(a), keeping the node numbers from the previous figure. Panel (a): isomorphism generated from flipping at node 0, i.e., applying $\gamma_0$. Panel (b): isomorphism generated from flipping at node 1, i.e., applying $\gamma_1$.}
\label{fig:treeIsomorphism}
\end{figure}

\subsection{Symmetry group of a tree and of options}
The set of isomorphisms of a given tree exhibit a group structure. The elements of the group are generated by the flips at internal nodes described above and the group operation is given by composition. It is clear that each flip is its own inverse, as exchanging left and right descendants of a node twice leaves the tree unchanged. Flips may be carried out in any order, so the operation is associative, and the identity is the element consisting of no flips.

When a tree $T$ is a parsing of a set of $n_o$ options, isomorphisms of the tree generate isomorphisms of the option set $\oneTo{n_o}$. Recall that an isomorphism of $T$ is a bijection $\varphi$ from the node set of $T$ to itself. Thus, a node $i$ is mapped to $j = \varphi(i)$ and the set $\{o(i): i \in \oneTo{n_o} \}$ is mapped to $\{ o(\varphi(i)): i \in \oneTo{n_o} \}$. The group of all possible isomorphisms of $n_o$ objects is $S_{n_o}$. Note, however, that not all such isomorphisms can be generated by the set of tree isomorphisms. For example, nodes that are siblings must remain siblings even under isomorphism operations.

Let the tree $T$ be a parsing of a set of $n_o$ options. The set of isomorphisms the options that can be generated by isomorphisms of $T$ forms a group $\Gamma_T \leq S_{n_o}$ whose structure is given by a wreath product of copies of $S_2$. This can be seen as follows. Let $i$ be an internal node in the tree $T$ and let $ro(i)$ and $lo(i)$ denote the set of options associated with its right and left descendants, respectively. For example, for the tree in Figure \ref{fig:treeIsomorphism}(a), $ro(0) = \{o(3), o(4), o(5)\}$ and $lo(0) = \{o(7), o(8) \}$. Let $\gamma_i$ represent the action of performing a flip at node $i$. Then $\gamma_i$ exchanges the sets $ro(i)$ and $lo(i)$. Explicitly, we have
\beq
\gamma_i : (ro(i), lo(i)) \mapsto (lo(i), ro(i)).
\eeq
Note that applying $\gamma_i$ twice results in the identity mapping, so $\gamma_i$ generates the permutation group $S_2$ acting on the set $\{ ro(i), lo(i) \}$. Furthermore, any options which are not associated with the descendants of node $i$ are unaffected. Thus, one can think of $\gamma_i$ as generating a representation of $S_2$ acting on the set $\oneTo{n_o}$; this is trivially a subgroup of $S_{n_o}$. Applying $\gamma_j$ for another internal node $j \neq i$ generates another representation of $S_2$. The group generated by application of both actions $\gamma_i$ and $\gamma_j$ is the wreath product $S_2 \wr S_2$. This process can be extended for each internal node $i$, yielding a symmetry group which is the repeated wreath product of copies of $S_2$. Formally we have the following Proposition.
\begin{proposition}[Symmetry group of a parsing] \label{prop:symmetryGroup}
Let the tree $T$ be a parsing of a set of $n_o$ options. Denote the number of internal nodes of $T$ by $n_i$. The symmetry group $\Gamma_T$ associated with the isomorphisms of $T$ is given by the $n_i$-fold wreath product of $S_2$
\beq
\Gamma_T = \underbrace{S_2 \wr S_2 \wr \cdots \wr S_2}_{n_i \text{ times}} \leq S_{n_o}.
\eeq
\end{proposition}


\section{A recursively-defined vector field} \label{sec:vectorField}
In this section we show how to construct a decision-making vector field for a given parsing of a finite set of options. We suppose we have a finite set of $n_o$ options and a tree $T$ which is a parsing of the options. Furthermore, each option $i$ is associated with a scalar $v_i > 0$ that encodes its importance or value. We seek a vector field $f$ operating on the state space $\Delta^{n_o}$ with attracting fixed points associated with the high-value options.

We label the $n_n$ nodes of $T$ with an index $i$, with the root node having the index $i=0$. We define the following notation to describe the tree structure. For a node $i$, we denote its parent by $\parent(i)$ and its left and right child nodes by $\lchild(i)$ and $\rchild(i)$, respectively. The descendants of a node $i$ consist recursively of the node's children $\lchild(i)$ and $\rchild(i)$ along with the descendants of the children. We denote the descendants of a node $i$ by $\desc(i)$. The descendants can be partitioned into left descendants and right descendants, which consist of the left child and its descendants and the right child and its descendants, respectively. For a node $i$, we denote the left descendants by $\ldesc(i)$ and the right descendants by $\rdesc(i)$, respectively. Recall that leaf nodes represent options. As above, let $o(i)$ be the option associated with a leaf node $i$.

To each node $i$ we associate the state $z_i \in \bbR$ and the value $u_i > 0$. Furthermore, to each internal (non-leaf) node $i$ we associate states $\bfz_i \in \bbR^2, \bfm_i \in \Delta^2$, and $\bfv_i \in \bbR_+^2$. These additional states represent quantities associated with the node's children. We denote the $j^{th}$ component of $\bfz_i, \bfm_i$, and $\bfv_i$ by $z_{ij}, m_{ij},$ and $v_{ij}$, respectively.

The states $z_i$ are defined recursively as follows. Let $z_0=1$. Then, for $i \geq 0, z_{lc(i)} = z_i m_{i1}$ and $z_{rc(i)} = z_i m_{i2}$. Alternatively, for $i \geq 1, z_i = z_{\parent(i)} m_{\parent(i)j}$, where
\[ j = \begin{cases}
1, & \text{if } i = \lchild(\parent(i)),\\
2, & \text{if } i = \rchild(\parent(i)).
\end{cases}\]
The state vector $\bfz_i$ associated with internal node $i$ has components $\bfz_i = (z_{\lchild(i)}, z_{\rchild(i)})^T = z_i \bfm_i$. Note that the definition of $z_i$ is invertible for $z_{\parent(i)} \neq 0$: in this case, we have $\bfm_i = \bfz_i/z_i$. The components are $m_{\parent(i)j} = z_i/z_{\parent(i)}$, where $j=1$ if $i = \lchild(\parent(i))$ and $j=2$ if $i = \rchild(\parent(i))$. 
Similarly, for a node $i$, $u_i$ is equal to the mean of the associated children's values
\beq \label{eq:internalV}
u_i = \begin{cases}
v_{o(i)}, & \text{if $i$ is a leaf}, \\
(v_{\lchild(i)} + v_{\rchild(i)})/2, & \text{if $i$ is an internal node.}
\end{cases}
\eeq
Then, for an internal node $i$, $v_{i1} = u_{\lchild(i)}$ and $v_{i2} = u_{\rchild(i)}$.

We endow the motivation states $\bfm_i$ associated with an internal node $i$ with dynamics $\dot \bfm_i = f(\bfm_i, \bfv_i)$, where $f$ is the Seeley \etal~dynamics \eqref{eq:miDynamics}. The dynamics of the overall tree $T$ consists of copies of the dynamics $f$ defined as follows. Recall that $n_i$ is the number of internal nodes of $T$ and let $\bfm \in \bbR^{2n_i}$ be the vector consisting of the stacked node states $\bfm_i$. Note that the definition of $\bfm$, i.e., the order in which the $\bfm_i$ are stacked, is arbitrary: different orders correspond to permutations of the coordinates. The structure of the dynamics is encoded in the tree structure, i.e., the parent-child relationships given by the functions $\parent, \lchild$, and $\rchild$. For practical purposes of performing computations, one chooses a scheme for numbering the coordinates that corresponds to a scheme for traversing the nodes of the tree. We choose to traverse the tree $T$ in a depth-first manner and define $\bfm$ by
\[ \bfm = (\bfm_0^T, \underbrace{\bfm_{\lchild(0)}^T, \bfm_{\lchild(\lchild(0))}^T, \ldots}_{\text{left descendants of node }0}, \underbrace{\bfm_{\rchild(0)}^T, \bfm_{\lchild(\rchild(0))}^T, \ldots}_{\text{right descendants of node }0})^T. \]
The dynamics of $\bfm$ are defined by stacking the dynamics of the component states $\bfm_i$:
\bal \label{eq:mDotTree}
\dot \bfm = f_m(\bfm, \bfv) &= (\dot\bfm_0^T, \dot \bfm_{\lchild(0)}^T, \dot \bfm_{\lchild(\lchild(0))}^T, \ldots, \dot \bfm_{\rchild(0)}^T, \dot \bfm_{\lchild(\rchild(0))}^T, \ldots)^T\\
&=(f(\bfm_0, \bfv_0)^T, f(\bfm_{\lchild(0)}, \bfv_{\lchild(0)})^T, f(\bfm_{\lchild(\lchild(0))}, \bfv_{\lchild(\lchild(0))})^T, \ldots, \nonumber \\
& \ \ \ \ \ \ \ \ \ \ \ \ \ \ \ \ \ \ \ \ \ f(\bfm_{\rchild(0)}, \bfv_{\rchild(0)})^T, f(\bfm_{\lchild(\rchild(0))}, \bfv_{\lchild(\rchild(0))})^T, \ldots)^T, \nonumber
\end{align}
where $\bfv \in \bbR^{2n_i}$ is the vector of the node value states $\bfv_i$ stacked in depth-first order.

The dynamics of $\bfm_i$ then defines dynamics of the states $\bfz_i$ by a simple change of coordinates. Recall that $\bfz_i = (z_{\lchild(i)}, z_{\rchild(i)})^T = z_i \bfm_i$ with $z_0 = 1$. Then $\dot \bfz_i = \dot z_i \bfm_i + z_i \dot \bfm_i = \dot z_i \bfm_i + z_i f(\bfm_i, \bfv_i)$. As above, we construct the vector $\bfz \in \bbR^{2n_i}$ by stacking the individual $\bfz_i$ in depth-first order. When $z_i \neq  0$, we have $\bfm_i = \bfz_i/z_i$, so $\dot \bfz_i$ can be written in terms of $\bfz$. We denote the resulting dynamics of $\bfz$ by
\beq \label{eq:zDot}
\dot \bfz = f_z(\bfz, \bfv).
\eeq
In the following section we show that the functions $f_m$ and $f_z$ are equivariant under changes of coordinates that correspond to isomorphisms of the tree $T$.

Let $\bfm_o \in \Delta^{n_o}$ be the state that represents the system's decision among the $n_o$ options. We relate the state $\bfz$ to $\bfm_o$ by projection onto $\bbR^{n_o+1}$. Let $o^{-1} : \oneTo{n_o} \to \oneTo{n_n}$ be the function that maps from an option $i$ to its associated leaf node.
\begin{definition}[Projected dynamics]
Let $T$ be a parsing of $n_o$ options, let $g$ be the dynamics defined by \eqref{eq:zDot}, and let $h: \bbR^{n_n} \to \bbR^{n_o+1}$ be the projection that reads off the elements of $\bfz$ that correspond to the leaves of the tree $T$. Explicitly, $h : \bfz \mapsto \bfm_o$ by 
\begin{equation} \label{eq:projection}
(\bfm_o)_i = m_{oi} = \begin{cases} 
z_{o^{-1}(i)}, & i \in \{ 1, 2, \ldots, n_o \}\\
1-\sum_{i=1}^{n_o} m_{oi}, & i = n_o + 1.
\end{cases}
\end{equation}
The projection $h$ then defines the dynamics of the projected state $\bfm_o$ by
\begin{equation} \label{eq:projectedDynamics}
\dot \bfm_o = f_o(\bfz, \bfv) = h(f_z(\bfz, \bfv)).
\end{equation}
\end{definition}
The projected dynamics \eqref{eq:projectedDynamics} leave the simplex $\Delta^{n_o}$ invariant.

\begin{theorem}
Let $T$ be a parsing of $n_o$ options, let $f_z$ be the dynamics defined by \eqref{eq:zDot}, and let $h: \bbR^{n_n} \to \bbR^{n_o+1}$ be the projection \eqref{eq:projection} that reads off the elements of $\bfz$ that correspond to the leaves of the tree $T$. Let $f_o$ be the dynamics defined by \eqref{eq:projectedDynamics}. Then $f_o$ leaves the simplex $\Delta^{n_o}$ invariant.
\end{theorem}
\begin{proof}
Let $\bfm$ be the projected state. Proving the claim reduces to showing that $m_i \geq 0$ and that $\sum_{i=1}^{n_o} m_i \leq 1$. We proceed by induction from the bottom of the tree. Let $i$ be a generic leaf node of $T$. By definition, $T$ is a proper binary tree, so $i$ has a sibling. Furthermore, since $i$ is a leaf node, it has a parent. Let $j$ denote the sibling of $i$ and $\parent(i)$ the common parent of $i$ and $j$. By the definition of the projection $h$, we have $m_{o(i)} = z_i, m_{o(j)} = z_j$. By the definition of $\bfz_i$, we have 
\[ \bfz_{\parent(i)} = (z_i, z_j) = z_{\parent(i)} \bfm_{\parent(i)} = z_{\parent(i)} (m_{\parent(i)1}, m_{\parent(i)2}). \]

Note that the dynamics \eqref{eq:miDynamics} of $\bfm_{\parent(i)}$ leave the simplex $\Delta^2$ invariant. This implies that the components $m_{\parent(i)j} \geq 0$ for $i = 1, 2$, and 
\beq \label{eq:miSumBound}
m_{\parent(i)1} + m_{\parent(i)2} \leq 1.
\eeq
Furthermore, we have that $z_i= m_{o(i)}$ and $z_j= m_{o(j)}$ are both non-negative.
Multiplying the expression \eqref{eq:miSumBound} by $z_{\parent(i)}$ yields the bound
\[ m_{o(i)} + m_{o(j)} = z_i + z_j \leq z_{\parent(i)}. \]
Thus, the sum of the $m_i$ associated with descendants of node $k = \parent(i)$ is upper bounded by $z_{k}$. The analogous argument holds for the parent of node $k$, and thus we can inductively work our way up the tree. At each node $l$, the sum of the $m_i$ associated with descendants of node $l$ is upper bounded by $z_l$.

The base case of the inductive argument is the root node $i=0$. All $n_o$ options are descendant from the root node, so we have $\sum_{i=1}^{n_o} m_i \leq z_0 = 1$, as desired.
\end{proof}

\section{The vector field is equivariant under tree isomorphisms}
Recall from Definition \ref{def:treeIsomorphism} above that two rooted trees $T_1$ and $T_2$ are said to be \emph{isomorphic} if there exists a bijection mapping between the nodes of $T_1$ and $T_2$ that preserves the root node. The vector field $f_z$ defined by \eqref{eq:zDot} and its projection $f_o$ defined by \eqref{eq:projectedDynamics} are equivariant under changes of coordinates which correspond to tree isomorphisms.

Let $T_1$ and $T_2$ be two isomorphic trees. By definition they must have the same number $n_n$ of nodes. The isomorphism between the trees is a bijection between the node sets of $T_1$ and $T_2$. In other words, it is a bijective map $\sigma : \oneTo{n_n} \to \oneTo{n_n}$. This is precisely the definition of a permutation.
\begin{definition} \label{def:permutation}
Let $T_1$ and $T_2$ be two isomorphic trees each with $n_n$ nodes. The map $\sigma : \oneTo{n_n} \to \oneTo{n_n}$ associated with the isomorphism between the trees is called the \emph{node permutation} corresponding to the isomorphism.
\end{definition}

The dynamics \eqref{eq:zDot} obey symmetries that correspond to isomorphisms of the underlying tree $T$. Formally, the dynamics \eqref{eq:zDot} are said to be \emph{equivariant}.
\begin{definition} \label{def:equivariance}
Let $X=\bbR^n$ and suppose that $\Gamma$ is a compact Lie group acting on $X$. Then a mapping $F: X \times \bbR \to X$ is \emph{$\Gamma$-equivariant} if and only if
\[ F(\gamma x, \lambda) = \gamma F(x, \lambda) \]
for all $\gamma \in \Gamma$, where $\lambda \in \bbR$ is a parameter.
\end{definition}

The Seeley~\etal~dynamics \eqref{eq:miDynamics} obey a symmetry that correspond to swapping the labels of the two options. When the option values are identical, the dynamics are $S_2$-equivariant.
\begin{lemma} \label{lem:seeleySymmetry}
Let $\pi_2 \in S_2$ represent the permutation of two elements. The Seeley~\etal~dynamics $f$ defined by \eqref{eq:miDynamics} are preserved under the action of $\pi_2$. Specifically, we have
\[ f(\pi_2 \bfm, \pi_2 \bfv) = \pi_2 f(\bfm, \bfv). \]
When $\bfv = (v, v)^T$, $f(\pi_2 \bfm, \bfv) = \pi_2 f(\bfm, \bfv)$, and the dynamics are $S_2$-equivariant.
\end{lemma}
\begin{proof}
The first statement is proven by straightforward substitution. From \eqref{eq:miDynamics}, we have
\[ f(\pi_2 \bfm, \pi_2 \bfv) = \begin{bmatrix}
v_2 (1-m_1-m_2) - m_2 \left( \frac{1}{v_2} - v_2 (1-m_1-m_2) + \sigma(m_1 m_2) \right) \\
v_1 (1-m_1-m_2) - m_1 \left( \frac{1}{v_1} - v_1 (1-m_1-m_2) + \sigma(m_1 m_2) \right)
\end{bmatrix}
= \pi_2 f(\bfm, \bfv).
\]
The second statement follows by noting that $\bfv = (v, v)^T$ implies $\pi_2 \bfv = \bfv$. Then, we have $f(\pi_2 \bfm, \bfv) = f(\pi_2 \bfm, \pi_2 \bfv) = \pi_2 f(\bfm, \bfv)$.
\end{proof}

Let $T$ be a parsing of $n_o$ options. Recall from Proposition \ref{prop:symmetryGroup} that the set of isomorphisms associated with $T$ form a group denoted $\Gamma_T$. These isomorphisms are represented by permutations $\sigma$. The dynamics \eqref{eq:zDot} obey symmetries corresponding to the permutations associated with $\Gamma_T$. When the option values are all identical, the dynamics are $\Gamma_T$-equivariant.

\begin{lemma} \label{lem:symmetry} 
The dynamics \eqref{eq:mDotTree} defined by a tree $T$ are preserved under isomorphisms of $T$. Explicitly, let $T$ be a parsing of $n_o$ options, $f_m$ be the dynamics \eqref{eq:mDotTree}, and let $\gamma \in \Gamma_T$. Then, $f_m(\gamma \bfm, \gamma \bfv) = \gamma f_m(\bfm, \bfv)$.
\end{lemma}
\begin{proof}
Let $\bfm$ represent the coordinates of \eqref{eq:mDotTree} that result from the default depth-first parsing of the tree $T$. Let $\gamma_i \in \Gamma_T$ represent the operation of flipping the tree $T$ at internal node $i$. Note that any $\gamma \in \Gamma_T$ can be represented as a composition of several flips $\gamma_i$, so it suffices to show that $f_m(\gamma_i \bfm, \gamma_i \bfv) = \gamma_i f_m(\bfm, \bfv)$ for any flip $\gamma_i$. 

Let $T_i^\prime$ be the tree that results from flipping $T$ at the internal node $i$, and let $\bfm^\prime$ represent the coordinates of \eqref{eq:mDotTree} that result from the depth-first traversal of the tree $T$. The dynamics \eqref{eq:mDotTree} take the form $\dot \bfm = f_m(\bfm, \bfv)$ in the coordinates associated with tree $T$ and $\dot \bfm^\prime = f_m(\bfm^\prime, \bfv^\prime)$ in the coordinates associated with $T^\prime$. Note that $\bfv^\prime$ represents $\bfv$ in the coordinates associated with $T^\prime$.

The action of $\gamma_i$ permutes the descendants of node $i$, and in particular swaps the right and left children of $i$: $\gamma_i : (m_{i1}, m_{i2}) \mapsto (m_{i2}, m_{i1})$. Compactly, this can be written as $\pi_2(m_{i1}, m_{i2}) = (m_{i2}, m_{i1}),$ where $\pi_2 \in S_2$ represents the permutation of two elements. The relation between $\bfm$ and $\bfm^\prime$ is as follows
\begin{align*} \bfm &= (\bfm_0^T, \ldots, \bfm_i^T, \underbrace{\bfm_{\lchild(i)}^T, \bfm_{\lchild(\lchild(i))}^T, \ldots}_{\text{left descendants of node } i}, \underbrace{\bfm_{\rchild(i)}^T, \bfm_{\lchild(\rchild(i))}^T, \ldots}_{\text{right descendants of node }i}, \ldots)^T \\
\bfm^\prime &= (\bfm_0^T, \ldots, \pi_2 \bfm_i^T, \underbrace{\bfm_{\rchild(i)}^T, \bfm_{\lchild(\rchild(i))}^T, \ldots}_{\text{right descendants of node } i}, \underbrace{\bfm_{\lchild(i)}^T, \bfm_{\lchild(\lchild(i))}^T, \ldots}_{\text{left descendants of node }i}, \ldots)^T,
\end{align*}
where $\lchild$ and $\rchild$ are the child relationships associated with tree $T$.

Recall that $\bfm_i = (m_{i1}, m_{i2})$ obeys the dynamics $\dot \bfm_i = f(\bfm_i,  \bfv_i)$ given by \eqref{eq:miDynamics}. By Lemma \ref{lem:seeleySymmetry}, we have $f(\pi_2 \bfm_i, \pi_2 \bfv_i) = \pi_2 f(\bfm_i, \bfv_i)$, so the action of $\gamma_i$ leaves the dynamics of $\bfm_i$ equivariant. It remains to study the action of $\gamma_i$ on the descendants of node $i$. As seen above, the action of $\gamma_i$ on these descendants is a block permutation, mapping $(\bfm_{\lchild(i)}, \bfm_{\rchild(i)}) \mapsto (\bfm_{\rchild(i)}, \bfm_{\lchild(i)})$, etc. The dynamics of each block $j$ is given by $\dot \bfm_j = f(\bfm_j, \bfv_j)$ and the overall dynamics $f_m$ is a simple stacking of copies of $f$. Since the action of $\gamma_i$ permutes both the blocks of $\bfm$ and $\bfv$ in the same way, the dynamics $f_m(\bfm^\prime, \bfv^\prime)$ consists of a permutation of the blocks of $f_m(\bfm, \bfv)$. Thus, we have
\[ f_m(\bfm^\prime, \bfv^\prime) = f_m(\gamma_i \bfm, \gamma_i \bfv) = \gamma_i f_m(\bfm, \bfv) \]
for any flip $\gamma_i$. The result follows by recalling that a generic $\gamma \in \Gamma_T$ can be represented by the composition of several flips $\gamma_i$.
\end{proof}

\begin{theorem} \label{thm:equivariance}
Let the conditions for Lemma \ref{lem:symmetry} be satisfied and suppose that $\bfv = v \mathbf{1}_{n_o}$, where $\mathbf{1}_{n_o} \in \bbR^{n_o}$ is the vector with all entries equal to 1, i.e., when $v_i = v \forall i \in \oneTo{n_o}$. Then, the dynamics $f_z$ defined by \eqref{eq:zDot} are equivariant under permutations of $\bfz$ corresponding to isomorphisms of $T$.
\end{theorem}
\begin{proof}
Let $\dot \bfz = f_z(\bfz, \bfv)$ be the dynamics \eqref{eq:zDot}. Let $\dot \bfm = f_m(\bfm, \bfv)$ be the dynamics \eqref{eq:mDotTree}. Note that the vector fields $f_z$ and $f_m$ are related by a change of coordinates $\bfz = g(\bfm)$ that is invertible away from the origin $\bfm = 0$. It is clear that $g(\gamma \bfm) = \gamma g(\bfm) \forall \gamma \in \Gamma_T$. Then, elementary calculus yields 
\beq 
f_z(\bfz, \bfv) = \dot \bfz = \frac{d}{dt} g(\bfm) = \frac{\partial g}{\partial \bfm} \dot \bfm = \frac{\partial g}{\partial \bfm} f_m(\bfm, \bfv) = \frac{\partial g}{\partial \bfm}(g^{-1}(\bfz)) f_m(g^{-1}(\bfz), \bfv).
\eeq

Analogously, we have
$f_z(\gamma \bfz, \gamma \bfv) = \frac{\partial g}{\partial \bfm}(g^{-1}(\gamma \bfz)) f_m(g^{-1}(\gamma \bfz), \gamma \bfv)$. The fact that $g(\gamma \bfm) = \gamma g(\bfm) = \gamma \bfz$ implies that $\gamma g^{-1}(\bfz) = \gamma \bfm = g^{-1}(\gamma \bfz)$. The chain rule yields $\frac{\partial g(\gamma \bfm)}{\partial \bfm} = \frac{\partial g(\gamma \bfm)}{\partial \bfm} \gamma$. Similarly, $g(\gamma \bfm) = \gamma g(\bfm)$ implies that $\frac{\partial g(\gamma \bfm)}{\partial \bfm} = \gamma \frac{\partial g(\bfm)}{\partial \bfm}$. Finally, note that $\gamma \cdot \gamma$ is equal to the identity for any $\gamma \in \Gamma_T$. Putting these facts together yields
\begin{align*}
f_z(\gamma \bfz, \gamma \bfv) &= \frac{\partial g}{\partial \bfm}(g^{-1}(\gamma \bfz)) f_m(g^{-1}(\gamma \bfz), \gamma \bfv)\\ 
&= \gamma \frac{\partial g^{-1}(\bfz)}{\partial \bfm} \gamma f_m(\gamma g^{-1}(\bfz), \gamma \bfv) \\
&= \gamma \frac{\partial g^{-1}(\bfz)}{\partial \bfm} \gamma \cdot \gamma f_m(g^{-1}(\bfz),  \bfv) = \gamma f_z(\bfz, \bfv).
\end{align*}

Thus, the dynamics $f_z$ obey the same tree isomorphism symmetry as the dynamics $f_m$. When $\bfv = v \mathbf{1}$, $\gamma \bfv = \bfv \forall \gamma \in \Gamma_T$. Then $f_z(\gamma \bfz, \bfv) = g(\gamma \bfz, \gamma \bfv) = \gamma g(\bfz, \bfv)$, the desired result.
\end{proof}

The implication of Lemma \ref{lem:symmetry} and Theorem \ref{thm:equivariance} is that the fundamental structure of the dynamics \eqref{eq:mDotTree} and \eqref{eq:zDot} is encoded in structure of the parsing $T$. Furthermore, when all the options have equal values, they are treated the same in the sense that by the dynamics of the corresponding states are unchanged by permutation of the coordinates. When the option values differ, however, these symmetries can be broken. The symmetry breaking can be understood by studying the bifurcation properties of the vector field.

\section{Bifurcation properties of the equivariant vector field} \label{sec:bifurcation}
The Seeley \etal~dynamics \eqref{eq:miDynamics} decide between two options using a pitchfork bifurcation that unfolds as the values of the two options differ. The dynamics \eqref{eq:mDotTree} and \eqref{eq:zDot} introduced in Section \ref{sec:vectorField} embed multiple copies of the pitchfork bifurcation inherited from \eqref{eq:miDynamics}. In this section we make this statement precise. We begin by recalling the definition of a $k$-parameter unfolding of a bifurcation.
\begin{definition}[\cite{MG-DGS:85}] \label{def:unfolding}
Let $f(x, \lambda) = 0$ be an equation which undergoes a bifurcation as $\lambda \in \bbR$ is varied. An \emph{unfolding} of $f$ is a parametrized family of functions $F(x, \lambda, \alpha), \alpha \in \bbR^k$, such that $F(x, \lambda, 0) = f(x, \lambda)$. One refers to $F$ as a $k$-parameter unfolding of $f$.
\end{definition}

We now recall the formal bifurcation result concerning the Seeley \etal~dynamics \eqref{eq:miDynamics}.

\begin{theorem}{\cite{TDS-etal:12, DP-etal:13}} \label{thm:SeeleyPitchfork}
Let $\dot \bfm = f(\bfm, \bfv)$ be the dynamics \eqref{eq:miDynamics} and let $\bfv = v \mathbf{1} \in \bbR_+^2$ be the vector with both entries equal to $v>1$. The dynamics undergo a pitchfork bifurcation as the parameter $\sigma$ increases through a critical value given by
\beq \label{eq:sigmaStar}
\sigma = \frac{4v^3}{(v^2-1)^2}.
\eeq
Equivalently, for fixed $\sigma$, the dynamics undergo a pitchfork bifurcation as the parameter $v$ increases through the critical value $v=v^*$ solving \eqref{eq:sigmaStar}.
\end{theorem}
\begin{proof}
When $\bfv = v \mathbf{1}$, straightforward computation shows that the dynamics $f$ defined by \eqref{eq:miDynamics} have an equilibrium $\bfm = \mbar \mathbf{1}$, where $\mbar$ satisfies
\beq \label{eq:mBarDeadlock}
\mbar = \frac{-(1 + v^2) + \sqrt{1+ 2 v^2 + 4 \sigma v^3 + 9 v^4}}{2 v (2 v + \sigma)}.
\eeq

Evaluating the Jacobian of $f$ yields 
\begin{align} \label{eq:mDotJacobian}
J 
&= \left. \begin{bmatrix}
-\frac{1}{v_1} - v_1 (1+m_1) + v_1 m_U - \sigma m_2 & -v_1(1+m_1)- \sigma m_1 \\ 
-v_2(1+m_2) & -\frac{1}{v_2} - v_2 (1+m_2) + v_2 m_U - \sigma m_1
\end{bmatrix}\right|_{(\bfm, \bfv) = (\mbar \mathbf{1}, v\mathbf{1})}\\
&= \begin{bmatrix}
-\frac{1}{v} - 3 v \mbar - \sigma \mbar & -v(1+\mbar) - \sigma \mbar \\ 
-v(1+\mbar) - \sigma \mbar & -\frac{1}{v} - 3 v \mbar - \sigma \mbar
\end{bmatrix}. \nonumber
\end{align}
The eigenvalues of $J$ are $\lambda_1 = \frac{-2 \mbar v^2+v^2-1}{v}, \lambda_2 = \frac{-4 \mbar v^2-2 \mbar \sigma  v-v^2-1}{v}$. Consider $\lambda_1$ and $\lambda_2$ as functions of $\sigma$. Simple substitution shows that $\lambda_2 < 0 \forall \sigma > 0$ and that $\lambda_1$ smoothly increases through zero as $\sigma$ increases through the value $\sigma^*$ defined by  \eqref{eq:sigmaStar}.
\end{proof}

As shown in Corollary \ref{cor:singularUnfolding}, the pitchfork bifurcation embedded in the $S_2$-equivariant dynamics \eqref{eq:miDynamics} unfolds as a function of a single parameter $\alpha = 2(v_1-v_2)/(v_1+v_2)$. Note that when $v_1 = v_2, \alpha = 0$. The bifurcation and unfolding properties of \eqref{eq:miDynamics} carry over naturally to the dynamics \eqref{eq:mDotTree}.

Note that the dynamics $f_m$ defined by \eqref{eq:mDotTree} consist of stacked copies of \eqref{eq:miDynamics}, so the Jacobian of $f_m$ is a block diagonal matrix whose diagonal blocks are copies of $J$ defined in \eqref{eq:mDotJacobian}. The singularity of $f_m$ then unfolds as a function of $n_i = n_o-1$ parameters $\alpha_i = 2(v_{i1}-v_{i2})/(v_{i1}+v_{i2})$. Formally, we have the following theorem.
\begin{theorem} \label{thm:treeJacobian}
Let $T$ be a parsing of $n_o$ options consisting of $n_i$ internal nodes. Let $\Gamma_T$ be the  isomorphism group of $T$. Let $\dot \bfm = f_m(\bfm, \bfv)$ be the dynamics \eqref{eq:mDotTree} defined by $T$ and let $v_i = v > 1$ for each option $i \in \oneTo{n_o}$ Then,
\renewcommand{\theenumi}{\roman{enumi})}
\begin{enumerate}
\item The vector field $f_m$ has an equilibrium $\bfm = \mbar \mathbf{1}_{2n_i}$, where $\mbar$ is defined by \eqref{eq:mBarDeadlock}.
\item The vector field has a singularity at $(\bfm, v, \sigma) = (\mbar \mathbf{1}_{2n_i}, v, \sigma)$, where $\sigma$ and $v$ are related by \eqref{eq:sigmaStar}. The singularity is a $\Gamma_T$-equivariant bifurcation that consists of $n_i$ copies of the standard $S_2$ pitchfork bifurcation.
\item When $\bfv$ is perturbed away from $\bfv = v \mathbf{1}$ the system $f_m(\bfm, \bfv)$ is a $n_o-1$-parameter unfolding of the $\Gamma_T$-equivariant bifurcation.
\end{enumerate}
\end{theorem}
\begin{proof}
Recall from \eqref{eq:mDotTree} that $f_m(\bfm, \bfv)$ consists of stacked copies of the dynamics $f(\bfm_i, \bfv_i)$ defined by \eqref{eq:miDynamics}. Since $v_i = v \forall i \in \oneTo{n_o}$, $\bfv_i = v \mathbf{1}_2$ and $\bfv = v \mathbf{1}_{2n_i}$. Thus, the $i^{th}$ block of $f_m(\bfm, \bfv)$ is equal to $f(\bfm_i, \bfv_i)$, which has an equilibrium $\bfm_i = \mbar \mathbf{1}_2$ as seen in the proof of Theorem \ref{thm:SeeleyPitchfork}. The equilibrium of $f$ follows by stacking the blocks $\bfm_i$, which proves statement i).

For statement ii), note that since the elements of $\bfm$ are stacked in the same order as those of $f_m(\bfm, \bfv)$, the Jacobian of $f_m(\bfm, \bfv)$ is a block diagonal matrix with the $i^{th}$ diagonal block being the Jacobian of $f(\bfm_i, \bfv_i)$. Thus, evaluating the Jacobian of $f_m(\bfm, \bfv)$ at the equilibrium in statement i) yields a block diagonal matrix
\[ J_m = \left.\frac{\partial f_m(\bfm, \bfv)}{\partial \bfm} \right|_{(\bfm, \bfv) = (\mbar \mathbf{1}_{2n_i}, v \mathbf{1}_{2n_i})} = 
\begin{bmatrix}
J & 0 & \ldots & 0 \\ 
0 & J & \ldots & 0 \\
\vdots & \vdots & \ddots & \vdots \\
0 & 0 & \ldots & J
\end{bmatrix} \in \bbR^{2 n_i \times 2 n_i}, \]
where $J$ is the matrix defined in \eqref{eq:mDotJacobian}. The eigenvalues of $J_m$ are $\lambda_1  = \frac{-2 \mbar v^2+v^2-1}{v}$ and $\lambda_2 = \frac{-4 \mbar v^2-2 \mbar \sigma  v-v^2-1}{v}$, each with multiplicity $n_i$. As shown in the proof of Theorem \ref{thm:SeeleyPitchfork}, $\lambda_1 = 0$ when $\sigma$ and $v$ are related by \eqref{eq:sigmaStar}, so there is a singularity at $(\bfm, v, \sigma) = (\mbar \mathbf{1}_{2n_i}, v, \sigma)$. This singularity consists of $n_i$ copies of the $S_2$ pitchfork bifurcation embedded in $J$.

For statement iii), consider how the option values $v_i, i \in \oneTo{n_o}$ are related to the value vector $\bfv \in \bbR^{2n_i}$. The vector $\bfv$ consists of $n_i$ blocks $\bfv_i$ whose components are defined by \eqref{eq:internalV}, one for each internal node. Note that, since $T$ is a full binary tree, it is a well-established fact \cite{MH:14} that $n_i = n_0 - 1$. Each $\bfv_i, i \in \oneTo{n_i}$ corresponds to an unfolding parameter $\alpha_i = 2(v_{i1}-v_{i2})/(v_{i1}+v_{i2})$. Each $\alpha_i$ is an unfolding parameter, since $\bfv = v \mathbf{1}$ implies that $\alpha_i$. The result follows by noting that $n_i = n_o-1$.
\end{proof}

The implication of this result is that the dynamics \eqref{eq:mDotTree} embeds a bifurcation which consists of multiple copies of the standard pitchfork bifurcation \eqref{eq:pitchfork}. The structure of the equilibria of \eqref{eq:mDotTree} in the post-bifurcation regime reflects the symmetry properties of the vector field and can be studied in detail by appeal to the equivariant branching lemma \cite[Theorem 3.3]{MG-IS-DGS:88}. The detailed analysis is beyond the scope of this paper, but we show numerical results in Section \ref{sec:numerics-bifurcation} below.




\section{Model reduction via singular perturbation} \label{sec:singularPerturbation}
The dynamics $f_m$ defined in \eqref{eq:mDotTree} inherit a complicated rational form from the Seeley~\etal~dynamics \eqref{eq:miDynamics}. As shown in Theorem \ref{thm:slowDynamics}, the dynamics \eqref{eq:miDynamics} can be reduced by singular perturbation. In this section, we carry out an analogous model reduction for \eqref{eq:mDotTree} and show that the equilibria of the resulting reduced model can be readily understood.

\subsection{Change of coordinates}
As in \cite{PBR:19d}, we apply singular perturbation theory to the dynamics \eqref{eq:mDotTree} by mapping $\bfv \mapsto K \bfv$ for a constant gain $K>0$ and take the singular limit $K \to +\infty$, or equivalently $\epsilon = 1/K \to 0$. The singular perturbation allows us to eliminate half of the state variables and thus to express equilibria in a straightforward way as a function of the option values $v_i$. The singular perturbation is more readily analyzed by expressing $\bfm_i = (m_{i1}, m_{i2}) \in \Delta^2$ and $\bfv_i = (v_{i1}, v_{i2}) \in \bbR^2_+$ in terms of mean-difference coordinates defined by
\[ \dm_i = m_{i1} - m_{i2}, \ \mb_i = \frac{m_{i1} + m_{i2}}{2}, \text{and }  \dv_i = v_{i1} - v_{i2}, \ \vb_i = \frac{v_{i1} + v_{i2}}{2}, \]
respectively. Expressing $\bfm_i$ in mean-difference coordinates results in expressing $\bfz_i \in \bbR^2$ in corresponding mean-difference coordinates
\[ (z_{i1}, z_{i2}) = \bfz_i = z_i \bfm_i = (z_i m_{i1}, z_i m_{i2})^T = \left( z_i \frac{2 \mb_i + \dm_i}{2}, z_i \frac{2 \mb_i - \dm_i}{2} \right)^T. \]

Note that the recursive definition of the $z_i$ coordinates as $\bfz_i = z_i \bfm_i$ is such that the value of $z_i$ can be expressed as the product of $\bfm_i$ along the path $p_i$ from the root to node $i$. Recall that we define $p_i$ as the sequence of nodes traversed along the unique shortest path from the root to node $i$. The sequence $p_i$ begins with the root node and ends with node $i$. We denote the number of nodes in the sequence by $|p_i|$, and the $j^{th}$ node of $p_i$ by $p_{ij}$. Explicitly, we have
\beq \label{eq:ziProduct}
z_i = \prod_{j =1}^{|p_i|-1} \frac{2 \mb_{p_{ij}} + a_j \dm_{p_{ij}}}{2},  \text{ where }  a_j = \begin{cases}
+1,& p_{i(j+1)} = \lchild(p_{ij}) \\
-1, & p_{i(j+1)} = \rchild(p_{ij}).
\end{cases}
\end{equation}

The dynamics of $z_i$ follow from the dynamics \eqref{eq:mDotTree} and can be reduced by applying singular perturbation theory. As in \cite{PBR-DEK:18, PBR:19d}, we map $\bfv$ to $K \bfv$, where $K>0$ is a constant gain. To apply singular perturbation theory, we set $\epsilon = 1/K$ to be a small parameter and define coordinates $x, y$ with components
\[ x_i = \dm_i, \ y_i = \frac{1-2\mb_i}{\epsilon}. \]

Since the dynamics $f_m$ defined in \eqref{eq:mDotTree} is composed of stacked copies of the dynamics $f$ defined in \eqref{eq:miDynamics}, singular perturbation of \eqref{eq:mDotTree} can be carried out by singularly perturbing its components which consist of copies of $f$. We can express the dynamics in the coordinates $(\dm_i, \mb_i)$ using the dynamics \eqref{eq:DeltamDot}, \eqref{eq:mBarDot}
\begin{align}
\dot \dm_i &= f_{\dm}(\dm_i, \mb_i; K \vb_i; K \dv_i) \\
\dot \mb_i &= f_{\mb}(\dm_i, \mb_i, \sigma; K \dv_i).
\end{align}
In the singular perturbation coordinates $(x_i, y_i)$, these dynamics become
\begin{align}
\dot x_i &= f_x(x_i, y_i; \dv_i, \vb_i, \epsilon) \label{eq:xDot} \\
&= -\epsilon \left( \frac{1-\epsilon y_i + x_i}{2\vb_i + \dv_i} - \frac{1-\epsilon y_i - x_i}{2 \vb_i - \dv_i} \right) + \vb_i x_i y_i + \dv_i y_i \frac{3-\epsilon y_i}{2} \nonumber \\
\epsilon \dot y_i &= g_y(x_i, y_i; \dv_i, \vb_i, \epsilon) \label{eq:yDot} \\
&= \epsilon \left( \frac{1-\epsilon y_i + x_i}{2\vb_i + \dv_i} - \frac{1-\epsilon y_i - x_i}{2 \vb_i - \dv_i} \right) + \frac{\sigma}{2}((1-\epsilon y_i)^2-x_i^2) \nonumber \\
&\ \ \ \ - \frac{y_i}{2} \left( (2\vb_i + \dv_i) \left( 1 + \frac{1 - \epsilon y_i + x_i}{2} \right) \right) -\frac{y_i}{2} \left( (2 \vb_i - \dv_i) \left(1 + \frac{1 - \epsilon y_i - x_i}{2}\right) \right). \nonumber
\end{align}

\subsection{Reduced node dynamics}
Taking the singular limit of the dynamics \eqref{eq:xDot}, \eqref{eq:yDot} associated with node $i$ yields a reduced system whose dynamics are given by a rational polynomial. This is formalized in the following theorem, which is a straightforward application of \cite[Theorem 1]{PBR:19d} stated above as Theorem \ref{thm:slowDynamics} and whose proof is reproduced here.
\begin{theorem} \label{thm:reducedNodeDynamics}
In the singular limit $\epsilon \to 0$, the dynamics \eqref{eq:xDot}, \eqref{eq:yDot} reduce to 
\beq \label{eq:xiDotSlow}
\dot x_i = \frac{\sigma}{2 \vb_i}(1-x_i^2) \frac{2x_i + 3 \alpha_i}{6 + \alpha_i x_i},
\eeq
where $\alpha_i = \dv_i/\vb_i$.
\end{theorem}

\begin{proof}
The proof follows the standard procedure for analyzing singularly-perturbed systems. First note that $x_i$ is the slow and $y_i$ the fast variable. Taking the singular limit $\epsilon \to 0$ of \eqref{eq:xDot} and \eqref{eq:yDot} yields
\begin{align}
\dot x_i &= f_x(x_i, y_i; \dv_i, \vb_i, 0) = \vb_i x_i y_i + \frac{3 \dv_i y_i}{2} \label{eq:xDotSingular} \\
0 &= g_y(x_i, y_i; \dv_i, \vb_i, 0) = -\frac{y_i}{2} \left( 6 \vb_i + \dv_i x_i \right) + \frac{\sigma}{2} \left( 1 - x_i^2 \right). \label{eq:yDotSingular}
\end{align}

Solving Equation \eqref{eq:yDotSingular} for the fast variable $y_i$ yields 
\[ y_i = h(x_i) := \frac{\sigma ( 1 - x_i^2 )}{6 \vb_i + \dv_i x_i} = \frac{\sigma}{\vb_i} \frac{1-x_i^2}{6+(\dv_i/\vb_i)x_i}, \]
which defines the slow manifold $\{ (x_i, y_i) = (x_i, h(x_i)) \}$. The system quickly converges to the slow manifold and then $x_i$ slowly evolves on the slow manifold. Using the expression $y_i = h(x_i)$ for the fast variable $y_i$ in terms of the slow variable $x_i$ yields the reduced slow dynamics 
\[ \dot x_i = f_x( x_i, h(x_i); \dv_i, \vb_i, 0) = \frac{\sigma}{2\vb_i}(1-x_i^2) \frac{2 x_i + 3 \dv_i/\vb_i}{6 + (\dv_i/\vb_i) x_i}. \]

Defining $\alpha_i = \dv_i/\vb_i$ yields the desired result \eqref{eq:xiDotSlow}.
\end{proof}

The implication of this result is that, in the singular limit $\epsilon \to 0$, $\mb_i \to 1/2$ and $\dm_i = x_i$ follows the dynamics \eqref{eq:xiDotSlow}. The coordinates of $\bfm_i$ associated with a node $i$ then reduce to 
\[ \bfm_i = \left( \frac{2 \mb_i + \dm_i}{2}, \frac{2 \mb_i - \dm_i}{2} \right)^T = \left( \frac{1+ \dm_i}{2} , \frac{1-\dm_i}{2} \right)^T, \]
where $\dm_i$ evolves according to \eqref{eq:xiDotSlow}. As shown in in Figure \ref{fig:equilibria}, the dynamics \eqref{eq:xiDotSlow} have equilibria $x_i = \pm 1, -3\alpha_i/2$ whose existence and stability properties depend on the unfolding parameter $\alpha_i$.

\subsection{Reduced tree dynamics}
The reduced dynamics of the full tree then follow from the reduction at each internal node $i$ introduced in Theorem \ref{thm:reducedNodeDynamics}. As noted above, the dynamics \eqref{eq:mDotTree} of the tree state $\bfm$ consists of stacked copies of the node dynamics $\dot \bfm_i = f_m(\bfm_i, \bfv_i)$. The reduced dynamics consist of stacked copies of the reduced dynamics \eqref{eq:xiDotSlow}. The equilibria of each state follow from the component dynamics. Formally, we have the following.

\begin{corollary} \label{cor:reducedTreeDynamics}
In the singular limit $\epsilon \to 0$, the dynamics \eqref{eq:mDotTree} reduce to 
\begin{equation} \label{eq:xDotSlowTree}
\dot \bfx = f_{\bfx}(\bfx, \bfv),
\end{equation}
where the $i^{th}$ component of $\bfx$ is equal to $x_i$, the states $x_i$ associated with each node $i$ are stacked in depth-first order, and $f_{\bfx}$ consists of stacked copies of the singularly-reduced dynamics \eqref{eq:xiDotSlow}.

The dynamics \eqref{eq:xDotSlowTree} have equilibrium states $\bfx_0$ whose $i^{th}$ component $x_{i,0}$ is equal to $\pm 1$ or $-3 \alpha_i/2$. The existence and stability properties of these equilibrium values depends on the unfolding parameter $\alpha_i$, as follows:
\[ x_{i,0} = \begin{cases}
+1, & \forall \alpha_i \in [-2,2], \text{ stable if } \alpha_i > -2/3 \\
-1, & \forall \alpha_i \in [-2,2], \text{ stable if } \alpha_i < 2/3 \\
-2\alpha_i/3, & \forall \alpha_i \in [-2/3, 2/3], \text{ unstable}.
\end{cases} \]
\end{corollary}

The equilibria of the reduced dynamics \eqref{eq:xDotSlowTree} imply a set of equilibria of the projected dynamics $f_o$ defined in \eqref{eq:projectedDynamics} whose values can be cleanly expressed in terms of the recursive definition \eqref{eq:ziProduct}. Formally, we have the following.
\begin{corollary} \label{cor:projectedSingularEquilibria}
Take the singular limit $\epsilon \to 0$ of the dynamics \eqref{eq:mDotTree} and consider the projection of these dynamics to the leaf states $\dot \bfm_o = f_o(\bfz, \bfv)$ given by \eqref{eq:projectedDynamics}. Denote the $i^{th}$ component of $\bfm_o$ as $m_{oi} = z_{o^{-1}(i)}$. The projected dynamics have equilibria with components
\begin{equation} \label{eq:projectedSingularEquilibria}
m_{oi} = z_{o^{-1}(i)} = \prod_{j =1}^{|p_i|-1} \frac{1 + a_j \dm_j^*}{2},
\end{equation}
where $a_j$ and $\dm_j^*$ are given by
\begin{align*}  a_j &= \begin{cases}
+1,& p_{o^{-1}(i)(j+1)} = \lchild(p_{o^{-1}(i)j}) \\
-1, & p_{o^{-1}(i)(j+1)} = \rchild(p_{o^{-1}(i)j})
\end{cases} \text{ and }\\
\dm_j^* &= \dm_{p_{o^{-1}(i)j}} = \begin{cases}
+1, & \forall \alpha_{p_{o^{-1}(i)j}} \in [-2,2], \\
-1, & \forall \alpha_{p_{o^{-1}(i)j}} \in [-2,2],\\ 
-2\alpha_{p_{o^{-1}(i)j}}/3, & \forall \alpha_{p_{o^{-1}(i)j}} \in [-2/3, 2/3].\\
\end{cases}
\end{align*}
Equilibria are stable if each $\dm_j^*$ is a stable equilibrium, where the stability of each $\dm_j^*$ is given in Corollary \ref{cor:reducedTreeDynamics}.
\end{corollary}
\begin{proof}
Recall that the value of $m_{oi} = z_{o^{-1}(i)}$ associated with an option $i$ can be expressed using the recursive definition \eqref{eq:ziProduct} in terms of $\mb_j$ and $\dm_j$ associated with nodes on the shortest path from the root node to the leaf node that represents the option $i$. In the singular limit $\epsilon \to 0$, we have $\mb_j \to 1/2 \forall j$ and $\dot \dm_j$ following the dynamics \eqref{eq:xDotSlowTree} with equilibria given in Corollary \ref{cor:reducedTreeDynamics}. Substituting the values of $\mb_j$ and $\dm_j$ yields the expression \eqref{eq:projectedSingularEquilibria}.

The stability property follows by contradiction. If any $\dm_j^*$ in the product \eqref{eq:projectedSingularEquilibria} corresponds to an unstable equilibrium of the singularly-perturbed dynamics \eqref{eq:xDotSlowTree}, then the overall equilibrium will be unstable.
\end{proof}

The implication of this result is clear from noting that the intermediate equilibrium $\dm_j^* = -2\alpha_{p_{o^{-1}(i)j}}/3$ is always unstable when it exists. Furthermore, the negative signs from $a_j$ and $\dm_j^*$ cancel out when the path $p_i$ passes from a parent to a right child. Thus, when the option value $v_i$ is sufficiently high so that $\alpha_j > -2/3$ whenever $p_i$ passes from a parent to a left child and that $\alpha_j < 2/3$ whenever $p_i$ passes from a parent to a right child. In this case, each element in the product \eqref{eq:projectedSingularEquilibria} is equal to one. Thus, when $v_i$ is sufficiently large relative to the other option values, the unique stable equilibrium of the projected dynamics \eqref{eq:projectedDynamics} is the state $\bfm_o = e_i$, where $e_i$ is the indicator vector with entry $i$ equal to 1 and all other entries equal to zero. Therefore, in the singular limit and when one option value is sufficiently large relative to the others, the dynamics \eqref{eq:projectedDynamics} carries out an $\arg \max$ operation on the value vector $\bfv$. When several option values are relatively large, the dynamics \eqref{eq:projectedDynamics} effectively performs a sort of dynamical $\arg \max$ operation whose output depends on initial conditions. See Section \ref{sec:numerics-singularPerturbation} for a numerical example.

\section{Numerical examples}
In this section we show the results of numerical simulations of the dynamics. All the computations have been carried out with code that is publicly available from the author's website \cite{PBR:20a}. The code is completely general in the sense that it implements the dynamics \eqref{eq:mDotTree} for a generic binary tree $T$. For clarity of presentation, all the simulations in this section are based on a binary tree parsing four options, as shown in Figure \ref{fig:4OptionsTree}. In all simulations, the parameter $\sigma$ in \eqref{eq:miDynamics} is set equal to 4 wherever it appears (once in \eqref{eq:mDotTree} for each internal node).

\begin{figure}[ht]
\centering
\begin{tikzpicture}[level/.style={sibling distance=40mm/#1}]
\node [circle, draw] (root){$0$}
	child {node [circle, draw] (a) {1}
		child {node [circle, draw] (b) {2}}
		child {node [circle, draw] (c) {3}}
	}
	child {node [circle, draw] (g) {4}
		child {node [circle, draw] (d) {5}}
		child {node [circle, draw] (e) {6}}
	};
\end{tikzpicture}\\
Ordered node list: $(0,1,2,3,4,5,6)$ \\
Ordered option list: (2,3,5,6) \\ 
\caption{Tree $T$ used in the simulations. $T$ is a parsing of four options, corresponding to the nodes $(2,3,5,6)$. The nodes are labeled with numbers according to the order in which they will be visited during a depth-first traversal.}
\label{fig:4OptionsTree}
\end{figure}
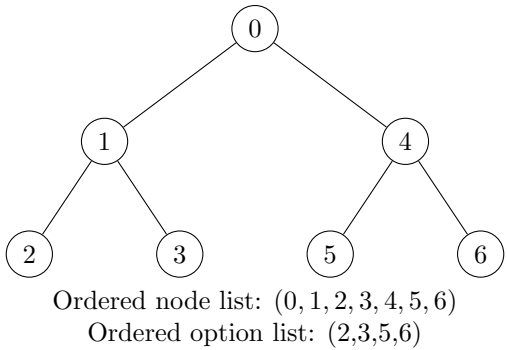

\subsection{Bifurcation characteristics} \label{sec:numerics-bifurcation}
Here we present the results of several simulations illustrating the bifurcation characteristics of the system as studied in Section \ref{sec:bifurcation}. Figures \ref{fig:deadlock} and \ref{fig:deadlockBroken} study the case of equal option values (i.e., $v_i = v$ for each option $i \in \oneTo{4}$) and show how the system \eqref{eq:mDotTree} bifurcates from having a single stable equilibrium to having five equilibria as the option value $v$ is increased past the critical value $v^* \approx 1.9058$ defined by \eqref{eq:sigmaStar}.

For the simulation shown in Figure \ref{fig:deadlock}, $v = 1.25 < 1.9058 \approx v^*$, so the system is in the pre-bifurcation regime. As predicted by Theorem \ref{thm:treeJacobian}, the dynamics \eqref{eq:mDotTree} have an equilibrium $\bfm = \mb \mathbf{1}$, where $\mb$ is defined by \eqref{eq:mBarDeadlock}. This equilibrium value of $\bfm$ gets projected to an equilibrium value $\bfm_o$ of \eqref{eq:projectedDynamics} whose $i^{th}$ component, corresponding to the $i^{th}$ option, is equal to $\mb^{d}$, with $d = |p_{o^{-1}(i)}|$ being the distance from the root of the tree $T$ to the $i^{th}$ option.

The simulation shown in Figure \ref{fig:deadlockBroken} is identical to that shown in Figure \ref{fig:deadlock}, except that now we set $v_i = v = 5$ so that the system is in the post-bifurcation regime. The four panels show trajectories resulting from four different initial conditions along with the (now unstable) deadlock equilibrium located at $\bfm_o = \mb \mathbf{1}$. As suggested by Corollary 11, in the post-bifurcation regime there is an additional set of four symmetric stable equilibria corresponding to a clear preference for each of the four equally-valued options. The equilibrium to which a trajectory is attracted depends on initial conditions. In panel (a), initial conditions were $\bfm = (\bfm_0^T, \bfm_1^T, \bfm_4^T)^T = (0.2, 0.1, 0.3, 0.2, 0.4, 0.2)^T$, corresponding to a weak initial preference for option 1. This initial preference determines the attracting equilibrium. The other three panels (b)--(d) use initial conditions that are permutations of those from panel (a). These permutations correspond to tree isomorphisms that exchange option 1 with options 2--4, respectively. As expected, the attracting equilibrium changes from option 1 to option 2--4 accordingly.

\begin{figure}[ht]
\centering
\includegraphics[width=0.7\textwidth]{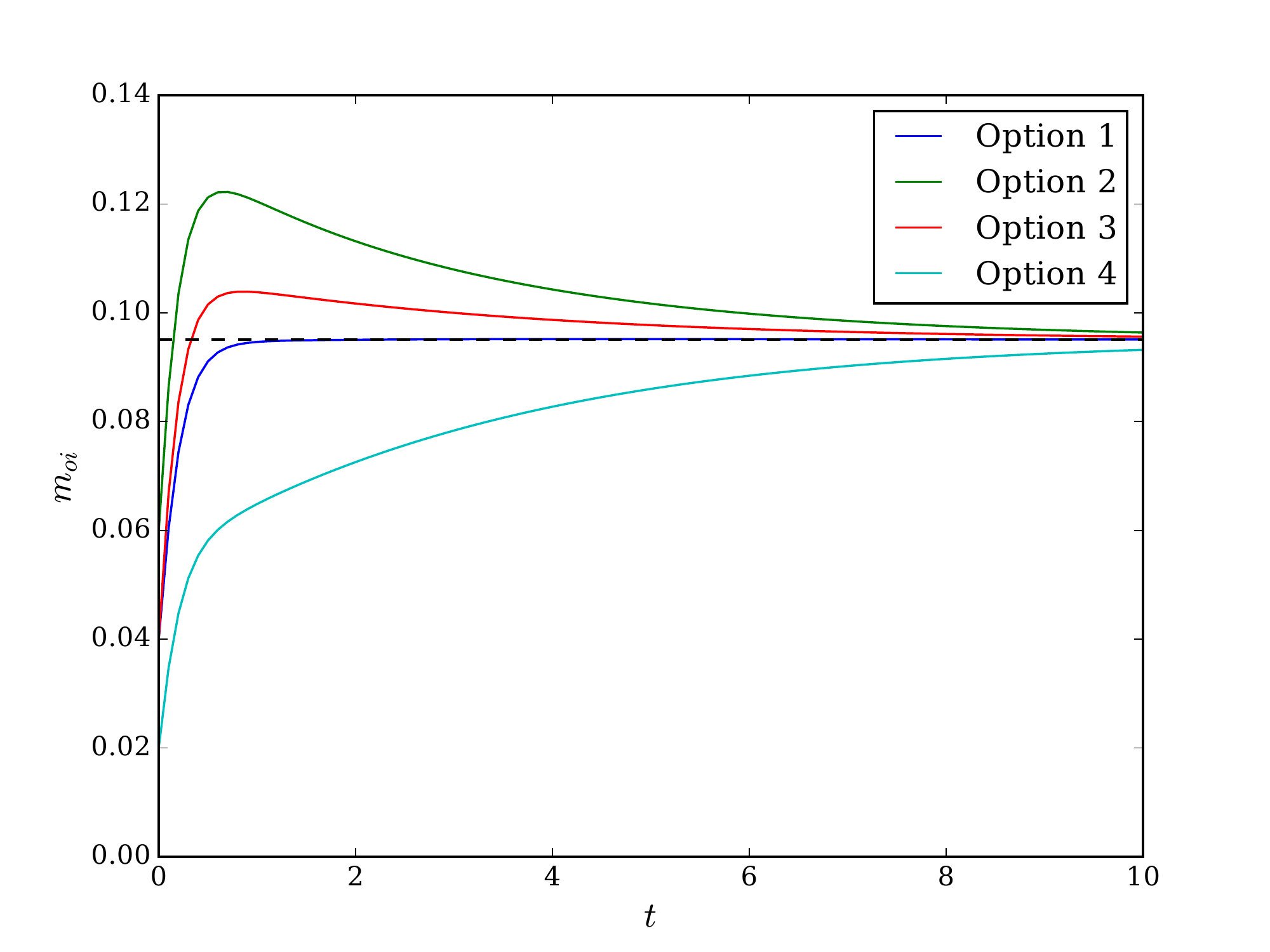}
\caption{Setting $v_i = v = 1.25$ for each option $i$ puts the system \eqref{eq:mDotTree} in the pre-bifurcation regime. This results in a stable deadlock equilibrium where no option wins. The traces show the trajectories of the components of the projected state $\bfm_o$, each of which converges to the deadlock value represented by the dashed line.}
\label{fig:deadlock}
\end{figure}

\begin{figure}[ht]
\centering
\begin{tabular}{cc}
\includegraphics[width=0.47\textwidth]{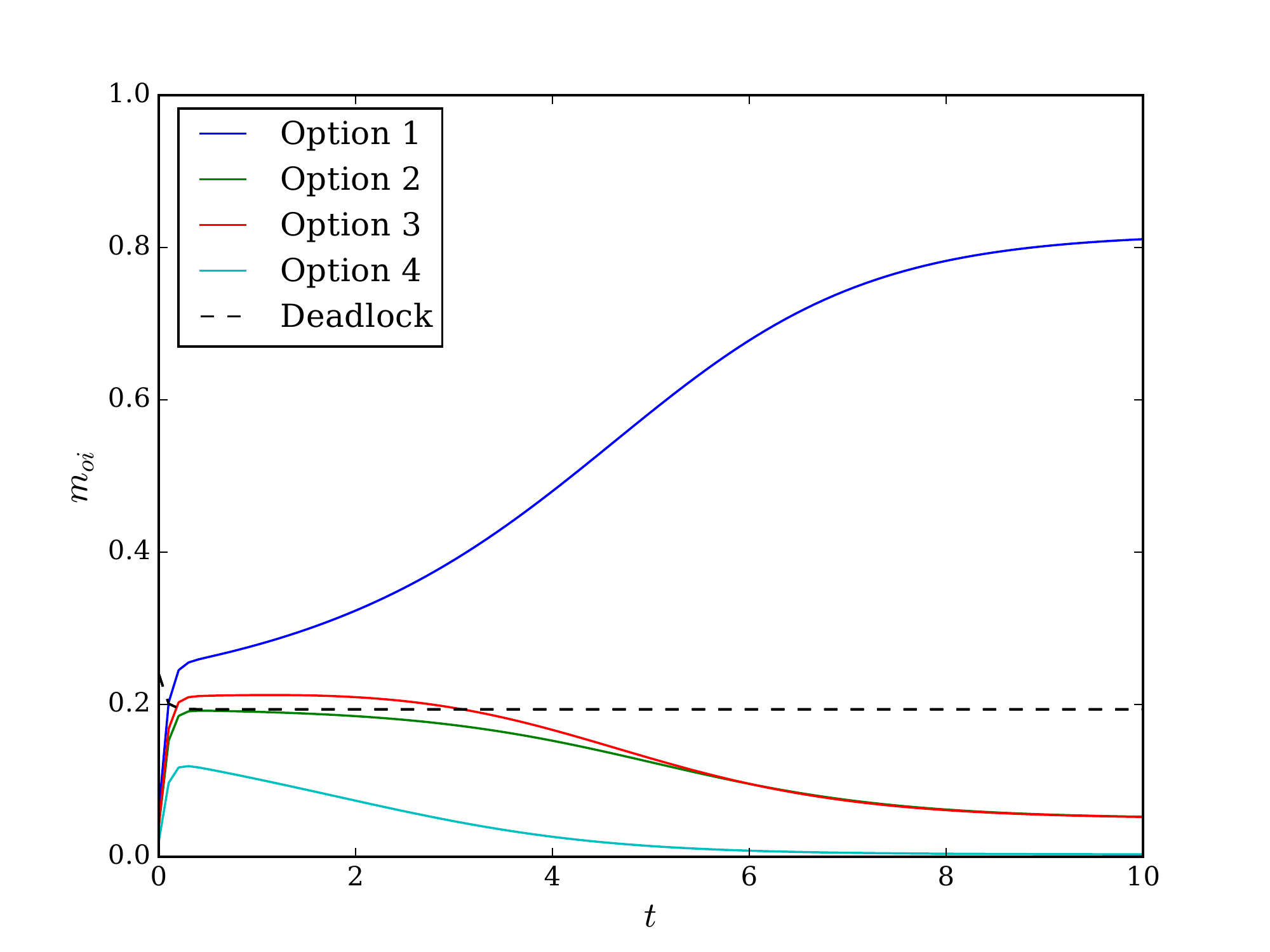} &
\includegraphics[width=0.47\textwidth]{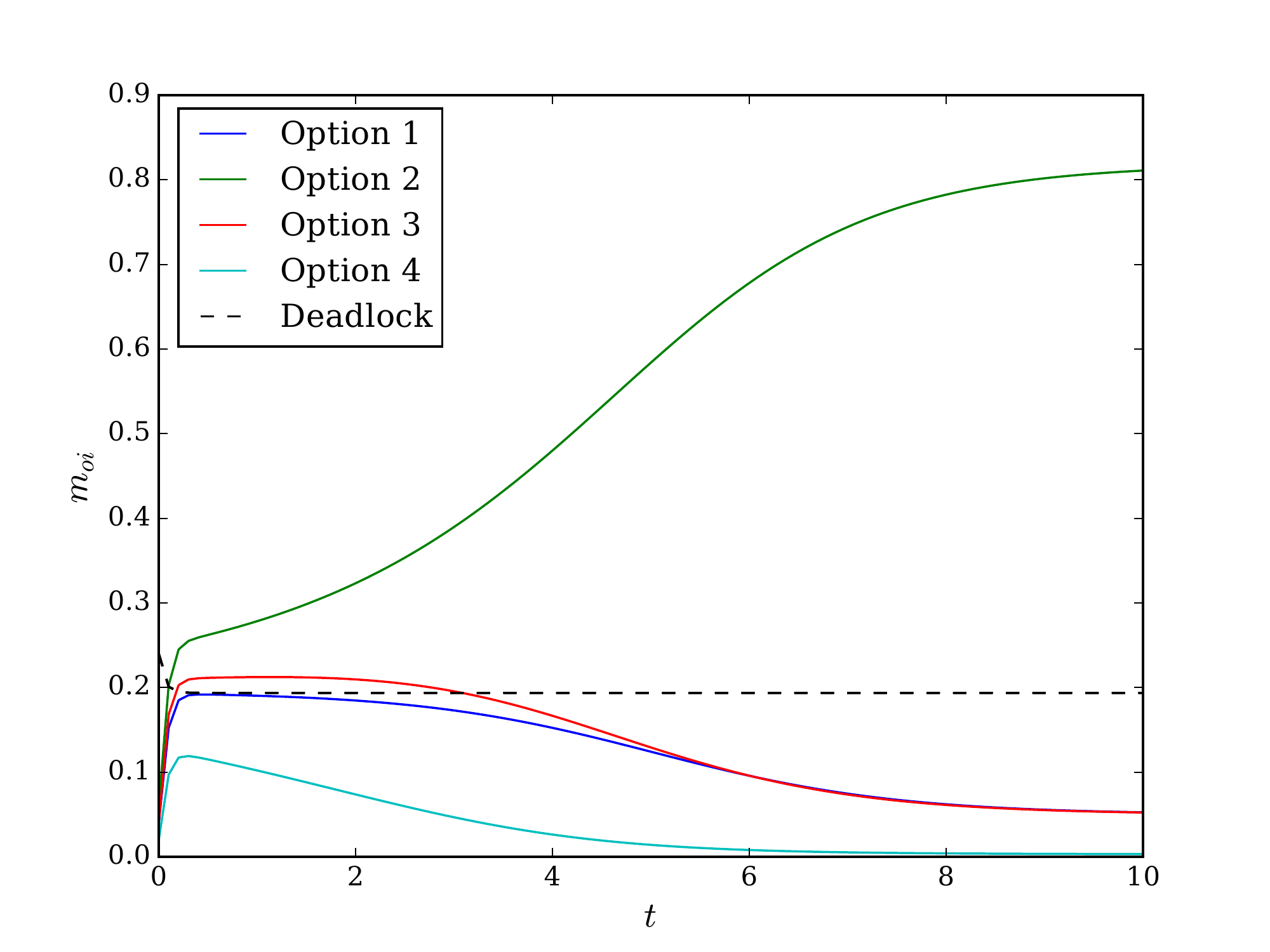} \\
\textbf{(a)} & \textbf{(b)} \\
\includegraphics[width=0.47\textwidth]{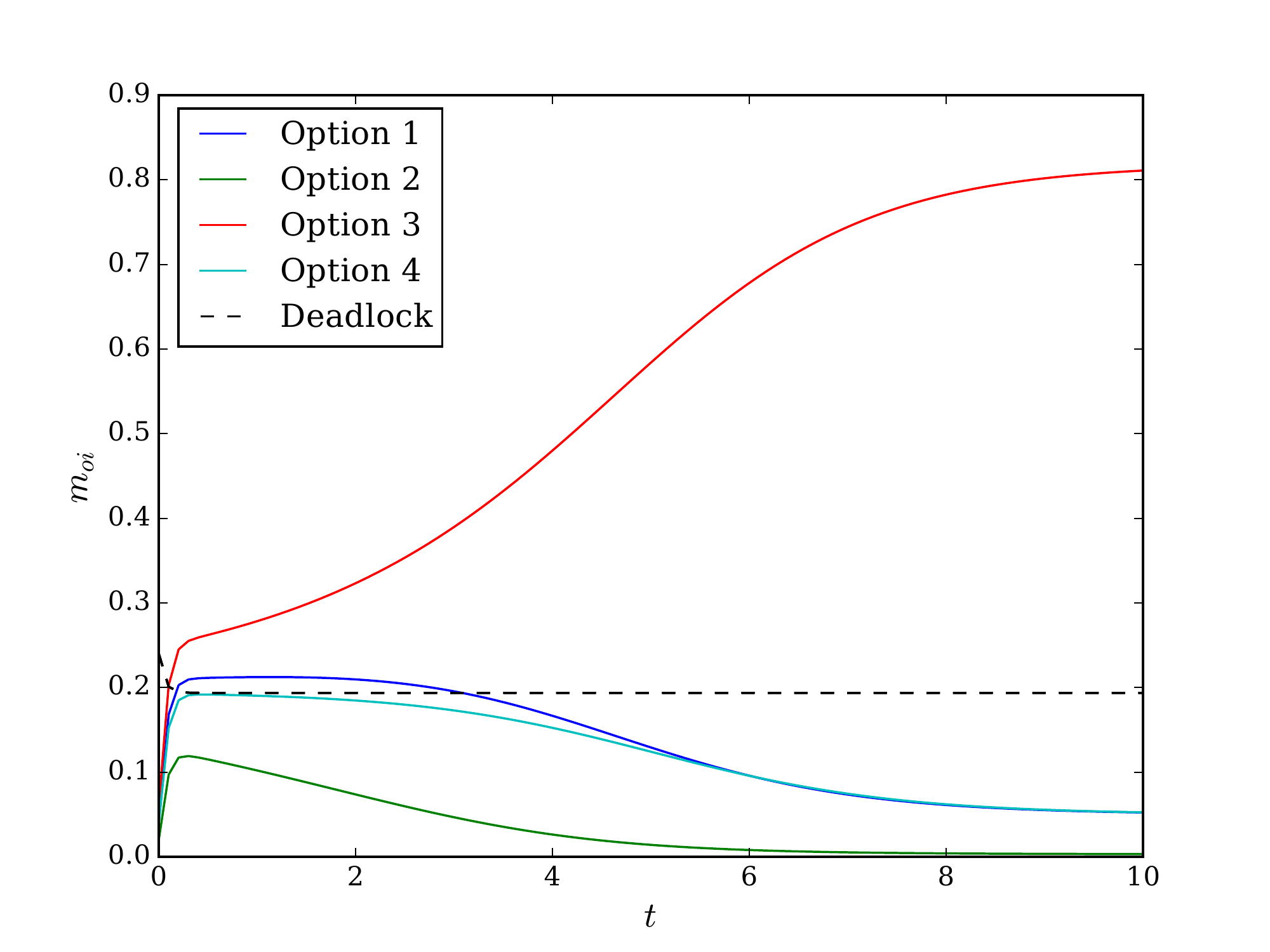} &
\includegraphics[width=0.47\textwidth]{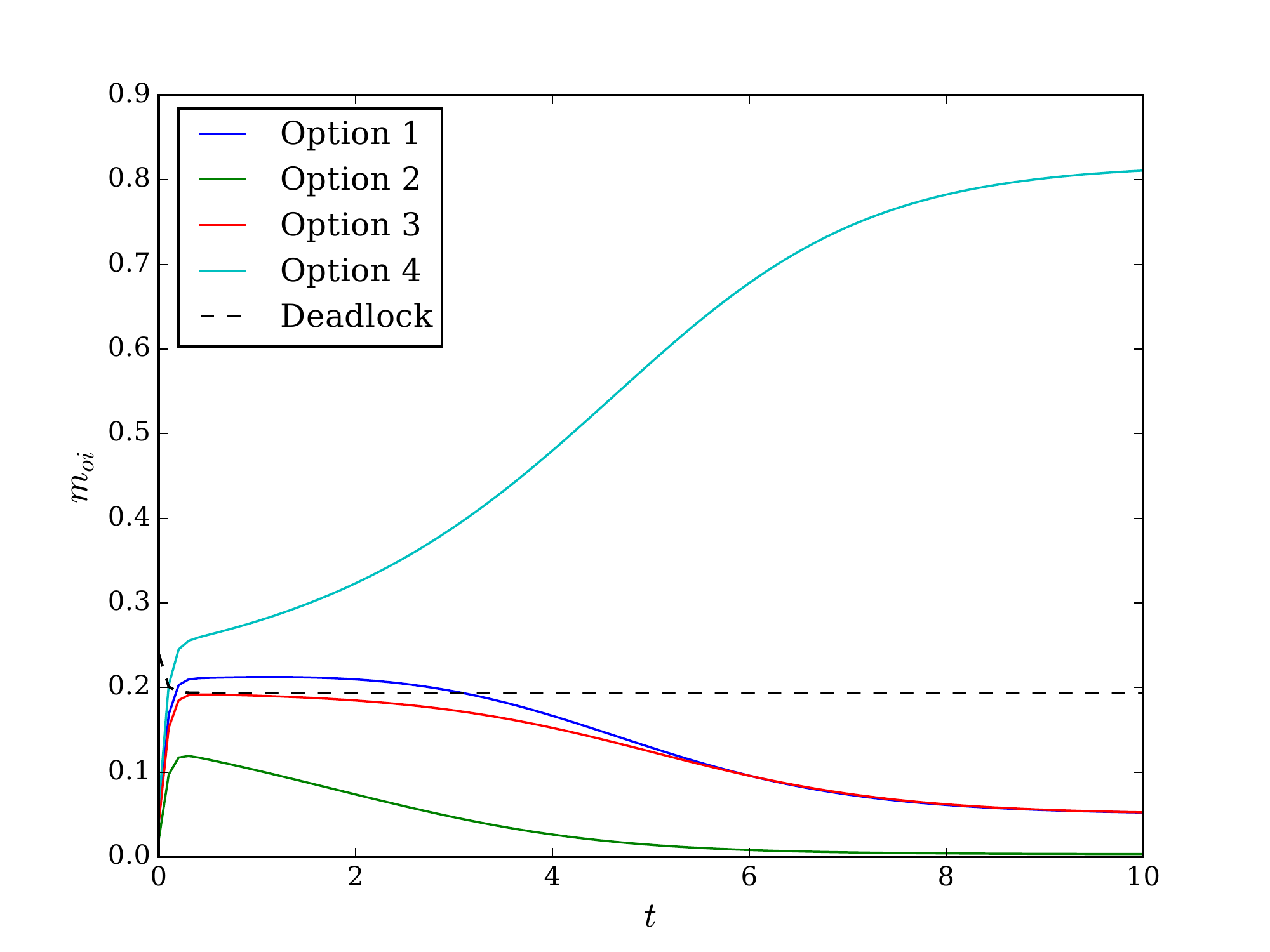} \\
\textbf{(c)} & \textbf{(d)}
\end{tabular}
\caption{Setting $v_i = v = 5$ for each option $i$ puts the system \eqref{eq:mDotTree} in the post-bifurcation regime. The traces show the trajectories of the components of the projected state $\bfm_o$. As suggested by Corollary \ref{cor:projectedSingularEquilibria}, there is a set of four symmetric stable equilibria corresponding to a clear preference for each of the four equally-valued options, along with an unstable deadlock equilibrium. The simulations producing the four plots are identical except for a permutation of the initial conditions that cause the system to converge to different stable states. Setting a symmetric initial condition $\bfz_0 = z \mathbf{1}$ where all initial state values are equal results in deadlock.}
\label{fig:deadlockBroken}
\end{figure}

\subsection{Singularly-perturbed system} \label{sec:numerics-singularPerturbation}
In Figure \ref{fig:singularlyPerturbed} we present the results of a simulation illustrating the results of Section \ref{sec:singularPerturbation}, particularly Corollary \ref{cor:projectedSingularEquilibria}. The values of the four options are set equal to $(100, 100, 300, 100)^T = 100(1,1,3,1)^T$, which puts the system close to the singularly-perturbed regime with $\epsilon = 1/100 \ll 1$. These option values are such that the unfolding parameters of the internal nodes are $\alpha_0 = -2/3, \alpha_1 = 0,$ and $\alpha_4 = 1$. In this case, Corollary \ref{cor:projectedSingularEquilibria} predicts that there should be a unique stable equilibrium at $\bfm_o = (0,0,1,0)^T$ corresponding to an absolute preference for the high-value option 3.

The results shown in Figure \ref{fig:singularlyPerturbed} confirm this prediction of a unique stable equilibrium, as the trajectories of the projected dynamics \eqref{eq:projectedDynamics} all converge to $\bfm_0 = (0,0,1,0)^T$ for four different simulations with initial conditions corresponding to the simulations shown in Figure \ref{fig:deadlockBroken}(a)--(d).

\begin{figure}[ht]
\centering
\includegraphics[width=0.7\textwidth]{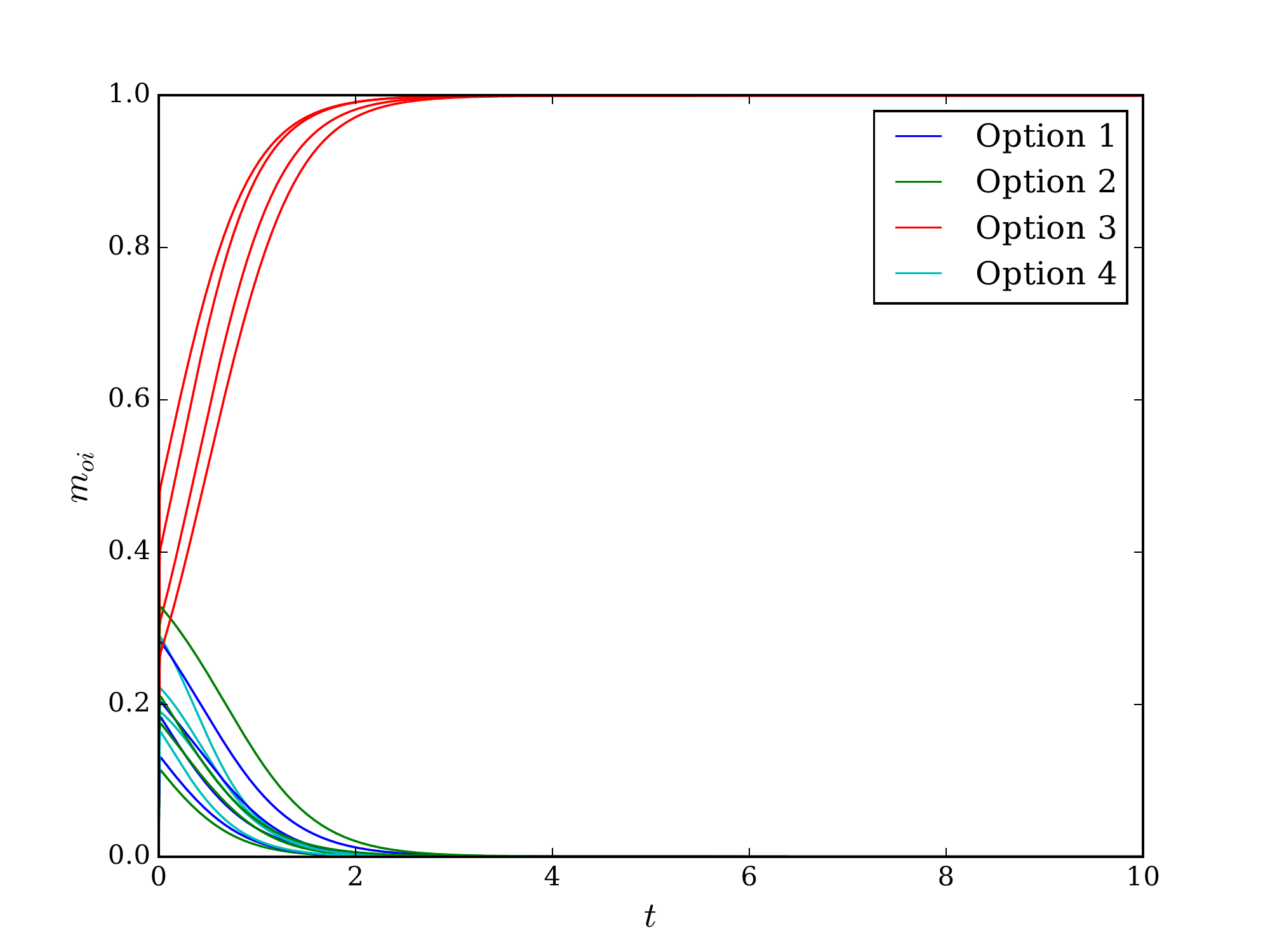}
\caption{Setting the option values $\bfv = (100, 100, 300, 100)^T = 100(1,1,3,1)^T$ puts the system \eqref{eq:mDotTree} close to the singularly-perturbed regime. Note that these option values are such that $\alpha_0 = -2/3$ and $\alpha_4 = 1$, so the only stable equilibrium should be the one corresponding to an absolute preference for the high-value option 3. As predicted by Corollary \ref{cor:projectedSingularEquilibria}, there is a single stable equilibrium corresponding to a preference for option 3. The traces show the trajectories of the components of the projected state $\bfm_o$ for the four different initial conditions considered in Figure \ref{fig:deadlockBroken}(a)--(d). In all cases, $\bfm_o$ converges to the state $(0,0,1,0)^T$.}
\label{fig:singularlyPerturbed}
\end{figure}

\section{Conclusion}
In this paper, we have developed a dynamical systems model of value-based decision making. The structure of the decision, and of the dynamical system itself, is encoded in a binary tree structure. The binary tree structure allows us to decompose a decision among $N$ options into a set of $N-1$ binary decisions arranged in a hierarchical structure. We then represent this decomposed decision as a dynamical system \eqref{eq:mDotTree} whose vector field is defined recursively by parsing down the binary tree. At each internal node of the tree, the system makes decisions based on the values associated with the node's two children, putting higher weight on the higher-value child. The $N$ leaf nodes of the tree represent the $N$ options.

The vector field \eqref{eq:mDotTree} has symmetries that correspond to isomorphisms of the underlying tree and associated option values. When the $N$ options all have the same value, all isomorphisms of the tree leave the vector field equivariant; when only some options have the same value, a smaller set of isomorphisms leave the vector field equivariant. The equilibria of the vector field have significant structure, organized around an $N-1$-parameter unfolding of a pitchfork singularity as shown in Theorem \ref{thm:treeJacobian}. The unfolding parameters of the pitchfork consist of the relative difference in values between the children of each internal node of the underlying tree. As shown in Corollary \ref{cor:projectedSingularEquilibria}, the system equilibria correspond to point attractors at states that correspond to a preference for the high-value option. In a singular limit, this preference becomes absolute in the sense that no weight is accorded to any other option.

Further work remains to be done to understand the structure of the symmetry group of the vector fields \eqref{eq:mDotTree} and \eqref{eq:projectedDynamics}, particularly in the case that only a subset of the options have identical values. In this case, the symmetry group will be a subgroup of the original group $\Gamma_T$, and such subgroups likely have interesting structure. Similarly, further work remains to be done to understand the structure of equilibria of the vector fields in the post-bifurcation regime. The main tool here is the equivariant branching lemma \cite[Theorem 3.3]{MG-IS-DGS:88}, which again leverages the subgroup structure of the symmetry group $\Gamma_T$.

There are a number of interesting questions raised by the binary tree structure of our model. For example, consider a generic case of deciding among $N>2$ options. The binary tree structure appears to be a strong constraint on the structure of the decision-making process. A more general value-based decision-making model, such as the one whose analysis was begun in \cite{AF-MG-NEL:19}, could have similar unfolding characteristics with fewer structural constraints. An open question is to understand the effect of the constraints imposed by the binary tree structure. Are there decisions that can be made by a model encoded as a flat graph (i.e., without a hierarchy structure) that cannot be made by our binary-tree-based model?

As discussed in the introduction, we anticipate the model developed in this paper to be valuable for a variety of problems requiring models of value-based decision-making behavior. We are actively pursuing applications in the area of control systems and robotics where options correspond to control vector fields and the present model affords a method to compose multiple such vector fields. In particular, we are developing methods to derive dynamics of the option values $v_i$ such that the overall system achieves a complex behavior specified, e.g., in terms of temporal logic. This work has the potential to unite dynamical systems with so-called formal methods tools \cite{HKG-ML-VR:18} for control.

\section*{Acknowledgement}
We thank Daniel Koditschek for discussions that led to the concept of a parsing. This work was supported in part by Air Force Research Laboratory grant FA8650-15-D-1845 subcontract 669737-6 and grant FA8650-19-C-1712 subcontract 670956-1.

\end{document}